\documentclass[11pt]{amsart}
\usepackage[margin=1in]{geometry}

\usepackage{amssymb}
\usepackage{amsthm}
\usepackage{amsmath}
\usepackage{mathrsfs}
\usepackage{amsbsy}
\usepackage[all]{xy}
\usepackage{bm}
\usepackage{hyperref}
\usepackage{tikz}
\usepackage{array}
\usepackage{float}
\usepackage{enumerate}
\usepackage{xcolor}
\usepackage{hhline}
\allowdisplaybreaks
\usepackage[noadjust]{cite}

\usepackage{caption}
\usepackage{tabu}
\usepackage{diagbox}

\usepackage[noabbrev,capitalise]{cleveref}

\newenvironment{enumerate*}%
  {\begin{enumerate}[(I)]%
    \setlength{\itemsep}{10pt}%
    \setlength{\parskip}{0pt}}%
  {\end{enumerate}}

\newtheorem{theorem}{Theorem}[section]
\newtheorem{proposition}[theorem]{Proposition}
\newtheorem{corollary}[theorem]{Corollary}
\newtheorem{conjecture}[theorem]{Conjecture}

\newtheorem{problem}[theorem]{Problem}
\newtheorem{lemma}[theorem]{Lemma}

\theoremstyle{definition}
\newtheorem{definition}[theorem]{Definition}
\newtheorem{remark}[theorem]{Remark}
\newtheorem{example}[theorem]{Example}

\definecolor{NormalGreen}{RGB}{0,220,0}
\definecolor{ChillBlue}{RGB}{0,182,255}
\definecolor{MyOrange}{RGB}{255,150,0}
\definecolor{LimeGreen}{RGB}{130,255,0}
\definecolor{MyPurple}{RGB}{180,0,255}

\newcommand{\dfn}[1]{\textcolor{blue}{\emph{#1}}}

\newcommand{\SortNoop}[1]{}

\usepackage{doi}

\newcommand{\Pop}{\mathsf{Pop}}

\newcommand{\cov}{\mathrm{cov}}
\newcommand{\Des}{\mathrm{Des}}
\newcommand{\Cells}{\mathrm{Cells}}
\newcommand{\MC}{\mathrm{MC}}
\newcommand{\spine}{\mathrm{spine}}
\newcommand{\Tam}{\mathrm{Tam}}
\newcommand{\DB}{\mathrm{DB}}
\newcommand{\Gal}{{\bf G}}
\newcommand{\Pos}{{\bf P}}
\newcommand{\cc}{\mathsf{c}}

\newcommand{\Heap}{\mathrm{Heap}}
\newcommand{\QQ}{\mathsf{Q}}
\newcommand{\word}{\mathsf{sort}}
\newcommand{\Camb}{\mathrm{Camb}}
\newcommand{\Av}{\mathrm{Av}}
\newcommand{\LRMax}{\mathrm{LRMax}}

\usepackage{etoolbox}

\newcommand{\includeSymbol}[1]{\ensuremath{%
	\mathchoice
		{\raisebox{-.7mm}{\includegraphics[height=2.2ex]{#1}}}	
		{\raisebox{-.7mm}{\includegraphics[height=2.2ex]{#1}}}
		{\raisebox{-.6mm}{\includegraphics[height=1.6ex]{#1}}}
		{\raisebox{-.5mm}{\includegraphics[height=1ex]{#1}}}
}}

\robustify{\includeSymbol}

\def\P{\mathbb{P}}

\def\bU{{\bf U}}

\usepackage{todonotes}

\begin{document}

\title{Ungarian Markov Chains}
\subjclass[2010]{}

\author[Colin Defant]{Colin Defant}
\address[]{Department of Mathematics, Massachusetts Institute of Technology, Cambridge, MA 02139, USA}
\email{colindefant@gmail.com}

\author[Rupert Li]{Rupert Li}
\address[]{Massachusetts Institute of Technology, Cambridge, MA 02139, USA}
\email{rupertli@mit.edu}

\maketitle

\begin{abstract}
We introduce the \emph{Ungarian Markov chain} ${\bf U}_L$ associated to a finite lattice $L$. The states of this Markov chain are the elements of $L$. When the chain is in a state $x\in L$, it transitions to the meet of $\{x\}\cup T$, where $T$ is a random subset of the set of elements covered by $x$. We focus on estimating $\mathcal E(L)$, the expected number of steps of $\bU_L$ needed to get from the top element of $L$ to the bottom element of $L$. Using direct combinatorial arguments, we provide asymptotic estimates when $L$ is the weak order on the symmetric group $S_n$ and when $L$ is the $n$-th Tamari lattice. When $L$ is distributive, the Markov chain $\bU_L$ is equivalent to an instance of the well-studied random process known as \emph{last-passage percolation with geometric weights}. One of our main results states that if $L$ is a trim lattice, then $\mathcal E(L)\leq\mathcal E(\text{spine}(L))$, where $\text{spine}(L)$ is a specific distributive sublattice of $L$ called the \emph{spine} of $L$. Combining this lattice-theoretic theorem with known results about last-passage percolation yields a powerful method for proving upper bounds for $\mathcal E(L)$ when $L$ is trim. We apply this method to obtain uniform asymptotic upper bounds for the expected number of steps in the Ungarian Markov chains of Cambrian lattices of classical types and the Ungarian Markov chains of $\nu$-Tamari lattices. 
\end{abstract}

\section{Introduction}\label{sec:intro}

\begin{figure}[ht]
  \begin{center}{\includegraphics[width=\linewidth]{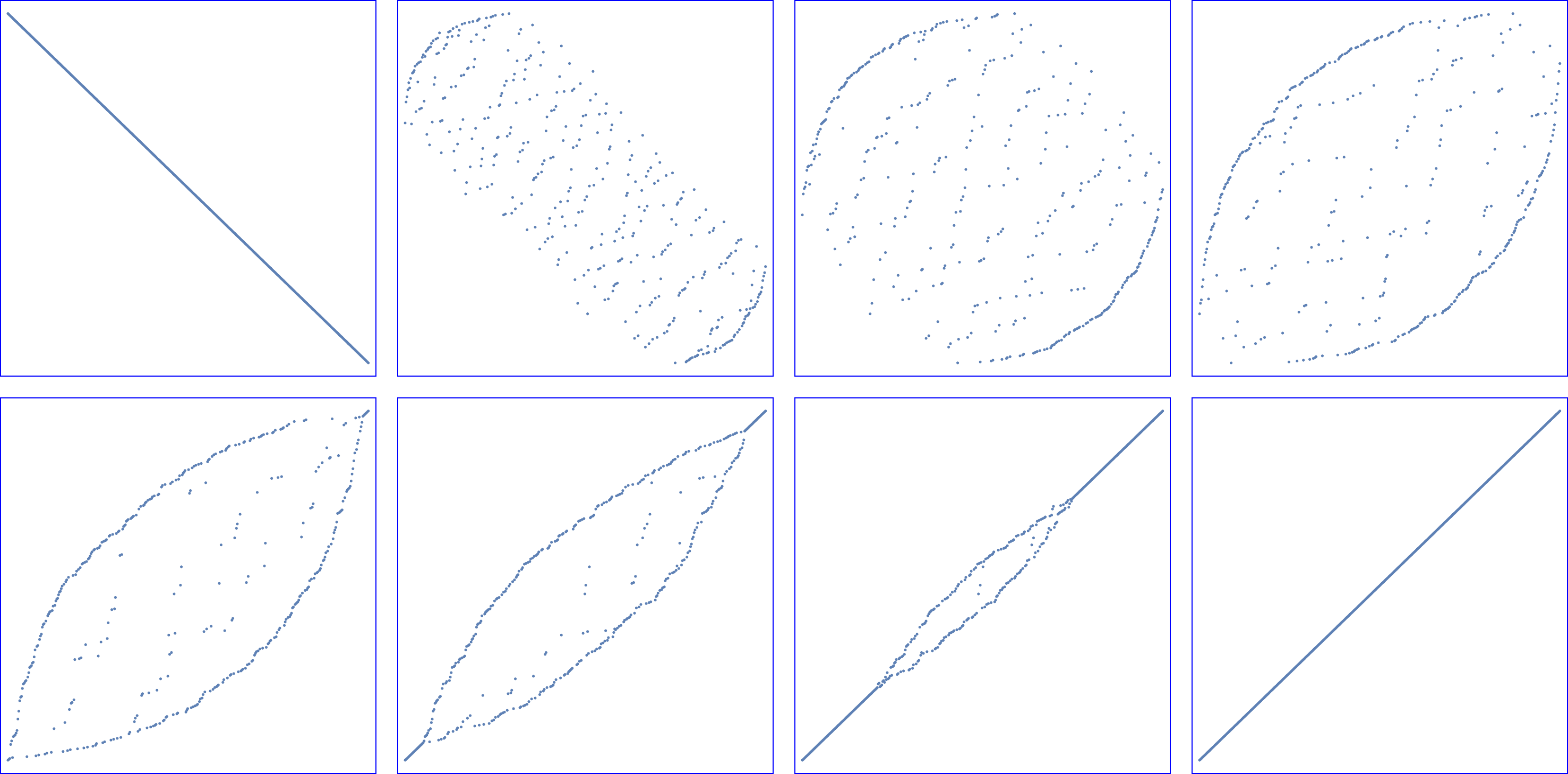}}
  \end{center}
  \caption{We ran the Ungarian Markov chain $\bU_{S_{400}}$ with $p=1/2$ starting with the decreasing permutation. Here, we show the plots of the permutations that the Markov chain reached at times $0,200,400,600,800,1000,1200,1400$ (read from left to right, with the top row before the bottom).}\label{fig:WeakPlots}
\end{figure}

\subsection{Ungar Moves} 
Given a set $\mathscr X$ of $n\geq 4$ points in the plane that do not all lie on a single line, one can consider the lines that pass through two or more of the points in $\mathscr X$. In 1970, Scott \cite{Scott} asked for the minimum possible number of distinct slopes determined by such a collection of lines. He observed that the answer is at most $n$ if $n$ is even (take $\mathscr X$ to be the set of vertices of a regular $n$-gon) and is at most $n-1$ if $n$ is odd (take $\mathscr X$ to be the set consisting of the vertices and the center of a regular $(n-1)$-gon). In 1979, Burton and Purdy \cite{BurtonPurdy} proved that the answer is at least $\left\lfloor n/2\right\rfloor$. Then, in 1982, Ungar \cite{Ungar} resolved the problem by showing that Scott's original upper bound is in fact the truth. His method, which followed a strategy proposed by Goodman and Pollack \cite{Goodman}, involved projecting the points in $\mathscr X$ onto a generic line. The ordering of the projected points along the line can be interpreted as an element of $S_n$, the set of permutations of the set $[n]=\{1,\ldots,n\}$. As one slowly rotates the line, the projected points will sometimes swap positions in the ordering. Thus, Ungar actually worked in a purely combinatorial setting in which he analyzed certain \emph{moves} that can be performed on permutations. 

Define an \dfn{Ungar move} to be an operation that reverses some disjoint collection of consecutive decreasing subsequences in a permutation. For example, we can perform an Ungar move on the permutation $763198542$ by reversing the consecutive decreasing subsequences $63$ and $542$ to obtain the permutation $736198245$. One can always perform a \dfn{trivial} Ungar move, which does nothing. At the other extreme, a \dfn{maximal} Ungar move reverses the maximal consecutive decreasing subsequences (also called \emph{descending runs}) of the permutation. Applying a maximal Ungar move to $763198542$ yields the permutation $136724589$. Let $\Pop\colon S_n\to S_n$ be the operator that applies a maximal Ungar move. This operator is called the \dfn{pop-stack sorting map} because it coincides with the map that passes a permutation through a data structure called a \emph{pop-stack} in a right-greedy manner. This map has already been studied in combinatorics and theoretical computer science \cite{Asinowski, Asinowski2, ClaessonPop, ClaessonPantone, Lichev, PudwellSmith}. If one starts with a permutation in $S_n$ and repeatedly applies the pop-stack sorting map, then one will eventually reach the identity permutation $12\cdots n$, which is fixed by $\Pop$. Ungar proved that the maximum number of iterations of $\Pop$ needed to send a permutation in $S_n$ to the identity is $n-1$. He also proved the following theorem, which we record for later use (see \cite[Chapter~12]{PftB} for further exposition on this result). 

\begin{theorem}[\cite{Ungar}]\label{thm:Ungar}
Let $n\geq 4$ be an integer. Suppose one starts with the decreasing permutation $n(n-1)\cdots 1$ and applies nontrivial Ungar moves until reaching the identity permutation $12\cdots n$. If the first Ungar move is not maximal, then the total number of Ungar moves used is at least $n$ if $n$ is even and is at least $n-1$ if $n$ is odd. 
\end{theorem}

\begin{remark}
The hypothesis that the first Ungar move is not maximal in \Cref{thm:Ungar} corresponds to the condition in the original geometric problem that the points in $\mathscr X$ are not collinear. 
\end{remark}

All posets in this article are assumed to be finite. A \dfn{lattice} is a poset $L$ such that any two elements $x,y\in L$ have a greatest lower bound---which is called their \dfn{meet} and denoted by $x\wedge y$---and a least upper bound---which is called their \dfn{join} and denoted by $x\vee y$. The meet and join operations are commutative and associative, so we can write $\bigwedge X$ and $\bigvee X$ for the meet and join, respectively, of any nonempty set $X\subseteq L$. Let $\cov_P(y)=\{x\in P:x\lessdot y\}$ denote the set of elements of a poset $P$ that are covered by an element $y\in P$.  

We can describe Ungar moves in a purely poset-theoretic manner if we consider the (right) weak order on $S_n$; this is a particular lattice that has been studied extensively in algebraic combinatorics. Applying an Ungar move to a permutation $x\in S_n$ is equivalent to taking the meet $\bigwedge(\{x\}\cup T)$ in the weak order for some set $T$ of elements covered by $x$. This allows us to vastly generalize the definition of Ungar moves to the setting of lattices as follows. 

\begin{definition}
Let $L$ be a lattice. An \dfn{Ungar move} is an operation that sends an element $x\in L$ to $\bigwedge (\{x\}\cup T)$ for some set $T\subseteq\cov_L(x)$. We say this Ungar move is \dfn{trivial} if $T=\emptyset$, and we say it is \dfn{maximal} if $T=\cov_L(x)$. 
\end{definition}

Let $L$ be a lattice. In \cite{DefantCoxeterPop, DefantMeeting}, the first author defined the \dfn{pop-stack sorting operator} $\Pop_L\colon L\to L$ to be the map that acts on an element of $L$ by applying a maximal Ungar move; that is, \[\Pop_L(x)=\bigwedge(\{x\}\cup\cov_L(x)).\] The articles \cite{ChoiSun, DefantCoxeterPop, DefantMeeting, Semidistrim, Hong, Sapounakis} study the dynamical properties and the image of $\Pop_L$ for various choices of interesting lattices $L$. 

\subsection{Random Ungar Moves} 
The goal of this article is to investigate what happens when one applies Ungar moves randomly. Let us fix a probability $p\in(0,1]$. Given an element $x$ of a lattice $L$, we choose a random subset $T\subseteq\cov_L(x)$ by adding each element of $\cov_L(x)$ to $T$ with probability $p$; we assume the choices for different elements are independent. We then apply the Ungar move that sends $x$ to $\bigwedge(\{x\}\cup T)$; let us call this a \dfn{random Ungar move} (the probability $p$ is implicit in this definition). This produces a Markov chain on $L$, a more symbol-heavy definition of which is as follows. 

\begin{definition}
Let $L$ be a lattice, and fix a probability $p\in (0,1]$. The \dfn{Ungarian Markov chain} of $L$ is the Markov chain $\bU_L$ with state space $L$ such that for all $x,y\in L$, we have the transition probability \[\P(x\to y)=\sum_{\substack{T\subseteq\cov_L(x) \\ \bigwedge(\{x\}\cup T)=y}}p^{|T|}(1-p)^{|\cov_L(x)|-|T|}.\]
\end{definition}

If we were to allow $p$ to be $0$, then $\bU_L$ would simply fix each element of $L$; we assume $p>0$ to avoid this boring scenario. If $p=1$, then $\bU_L$ is deterministic and agrees with the operator $\Pop_L$. 

\begin{remark}
One could define a more general version of an Ungarian Markov chain by assigning a probability $p_e$ to each edge $e$ in the Hasse diagram of $L$. However, for the sake of simplicity, we will only consider the case in which there is a single probability $p$. 
\end{remark}

The lattice $L$ has a unique minimal element $\hat 0$ and a unique maximal element $\hat 1$. The Markov chain $\bU_L$ is absorbing, and its unique absorbing state is $\hat 0$. We will be primarily interested in estimating the expected number of steps required to reach the absorbing state $\hat 0$ when we start at $\hat 1$ in $\bU_L$. We denote this expected number of steps by $\mathcal E(L)$. Note that $\mathcal E(L)$ implicitly depends on the fixed probability $p$. 
\begin{remark}
The first author and Williams \cite{Semidistrim} recently introduced \emph{semidistrim} lattices, which form a broad class of finite lattices that generalize distributive, semidistributive, and trim lattices. They showed how to define a natural invertible \emph{rowmotion} operator on any semidistrim lattice, and they found that rowmotion is closely related to pop-stack sorting. In some sense, pop-stack sorting can be viewed as the ``non-invertible cousin'' of rowmotion. In a recent article \cite{DefantLiRowmotion}, we introduced \emph{rowmotion Markov chains} by inserting randomness into the definition of rowmotion, and we proved that the rowmotion Markov chain of a semidistrim lattice $L$ is always irreducible (assuming $0<p<1$). One can think of the Ungarian Markov chain of $L$ as the ``absorbing cousin'' of the rowmotion Markov chain of $L$. 
\end{remark}
\begin{remark}
Suppose $L$ is a semidistrim lattice and $u,v\in L$ are such that $u\leq v$. Using results from \cite{Semidistrim}, one can show that the order complex of the open interval $(u,v)=[u,v]\setminus\{u,v\}$ is contractible if and only if $u$ can be obtained by applying an Ungar move to $v$. Thus, if one starts at $\hat 1$ and runs the Markov chain $\bU_L$ until reaching $\hat 0$, then the set of states visited throughout the process will form an \emph{essential chain} of $L$ in the sense of Bj\"orner \cite{BjornerEssential}. (Technically, we must assume that at least one state other than $\hat 0$ and $\hat 1$ is visited in order to match Bj\"orner's definition of an essential chain.) 
\end{remark}

\subsection{Symmetric Groups Under the Weak Order}
We view the symmetric group $S_n$ as a lattice under the weak order. A \dfn{descent} of a permutation $x\in S_n$ is an index $i\in[n-1]$ such that $x(i)>x(i+1)$; let $\Des(x)$ be the set of descents of $x$.
To apply a random Ungar move to a permutation $x$, we choose a subset of $\Des(x)$ by including each descent with probability $p$, and we reverse the consecutive decreasing subsequences of $x$ corresponding to those descents. For example, the set of descents of $731496852$ is $\{1,2,5,7,8\}$; if we choose the subset $\{1,7,8\}$ (this happens with probability $p^3(1-p)^2$), then we obtain the new permutation $371496258$. 

\begin{remark}
Applying a random Ungar move to a permutation $x$ is equivalent to sending $x$ through a pop-stack (see \cite{Avis, SmithVatter} for the definition) using the following two rules: 
\begin{itemize}
\item the entries in the pop-stack must always be increasing from top to bottom; 
\item at each point in time when we have a choice whether to pop the entries out of the stack, we do so with probability $1-p$. 
\end{itemize}
\end{remark}

There is also an algebraic way to interpret an Ungar move if we think of $S_n$ as the Coxeter group generated by $\{s_1,\ldots,s_{n-1}\}$, where $s_i$ is the transposition $(i\,\,i+1)$.
Suppose we apply an Ungar move to a permutation $x$ by choosing a subset $T\subseteq\Des(x)$ and reversing the corresponding decreasing subsequences. The result is the permutation $xw_0(T)$, where $w_0(T)$ is the maximal element (in the weak order) of the parabolic subgroup of $S_n$ generated by $\{s_i:i\in T\}$.
The element $w_0(T)$ is an involution.
Thus, if we start at a permutation $x$ and run the Markov chain $\bU_{S_n}$ until reaching the identity permutation $12\cdots n$ (which is the absorbing state), we will generate a random factorization $x_1\cdots x_m$, where each factor $x_k$ is of the form $w_0(T)$ for some $T\subseteq[n-1]$. This factorization is \emph{length-additive} in the sense that the Coxeter length (i.e., the number of inversions) of $x$ is equal to the sum of the Coxeter lengths of the factors $x_1,\ldots,x_m$. 

The \dfn{plot} of a permutation $x\in S_n$ is the graph showing the points $(i,x(i))$ for all $i\in[n]$. \Cref{fig:WeakPlots} shows the plots of eight of the permutations that we obtained while running $\bU_{S_{400}}$ with $p=1/2$ starting with the decreasing permutation. It would be extremely interesting to determine the shapes of these permutations in a manner similar to what was done in \cite{Angel, Angel2, Dauvergne, DauvergneVirag} for permutations obtained via random sorting networks.  

It follows easily from Ungar's \Cref{thm:Ungar} that if we fix the probability $p\in(0,1)$, then $\mathcal E(S_n)$, the expected number of steps in the Ungarian Markov chain $\bU_{S_n}$ needed to go from the decreasing permutation to the identity permutation, is at least $n-1+o(1)$; see \Cref{lem:lower_weak} below.
Noting that $S_n$ is a graded lattice of total rank $\binom{n}{2}$, one can easily show that $\mathcal E(S_n)$ grows at most quadratically in $n$. The following theorem greatly improves upon this naive upper bound. 

\begin{theorem}\label{thm:weak}
For $p\in(0,1)$ fixed, we have \[n-1+o(1)\leq\mathcal E(S_n)\leq \frac{8}{p}n\log n+O(n)\] as $n\to\infty$. 
\end{theorem}

\subsection{Distributive Lattices}\label{subsec:distributive}
An \dfn{order ideal} of a poset $P$ is a set $Y\subseteq P$ such that if $x,y\in P$ satisfy $x\leq y$ and $y\in Y$, then $x\in Y$. When ordered by inclusion, the order ideals of $P$ form a distributive lattice that we denote by $J(P)$. Birkhoff's Fundamental Theorem of Finite Distributive Lattices \cite{Birkhoff} states that every (finite) distributive lattice is isomorphic to the lattice of order ideals of some (finite) poset. 

Let $P$ be a poset, and consider an order ideal $I\in J(P)$. The set $\cov_{J(P)}(I)$ consists of the order ideals of $P$ that can be obtained by removing a single maximal element from $I$. Therefore, applying a random Ungar move to $I$ is equivalent to removing a random subset of the set of maximal elements of $I$, where each maximal element is removed with probability $p$. Associate to each $x\in P$ a geometric random variable $G_x$ with parameter $p$ (i.e., with expected value $1/p$), and assume that the random variables associated to different elements of $P$ are all independent. If we start the Markov chain $\bU_{J(P)}$ at the top element $\hat 1=P$, then we can think of $G_x$ as the number of steps throughout the process during which the state (i.e., order ideal) has $x$ as a maximal element. Thus, the number of steps needed to go from the top element $\hat 1=P$ to the bottom element $\hat 0=\emptyset$ is $\max\limits_{\mathcal C\in\MC(P)}\sum_{x\in \mathcal C}G_x$, where $\MC(P)$ is the set of maximal chains of $P$. Consequently, 
\begin{equation}\label{eq:distributive_max_chains}
\mathcal E(J(P))=\mathbb E\left(\max_{\mathcal C\in\MC(P)}\sum_{x\in \mathcal C}G_x\right).
\end{equation}

The above reformulation of the Markov chain $\bU_{J(P)}$ in terms of geometric random variables is a well-studied random process known as \dfn{last-passage percolation with geometric weights}. Therefore, while we originally viewed Ungarian Markov chains as randomized versions of pop-stack sorting, one could equally well package them as generalizations of last-passage percolation with geometric weights.

Last-passage percolation has been investigated thoroughly when $P$ is the poset associated to a Young diagram, where it has also been called the \emph{multicorner growth process} \cite[Chapters~4~\&~5]{Romik} and is very closely related to the \emph{totally asymmetric simple exclusion process} (TASEP) \cite{Rost}. In this setting, very precise results are known concerning the number of steps needed to reach the absorbing state $\hat 0$ and even what the order ideals (i.e., Young diagrams) typically ``look like'' throughout the process. We refer the reader to \cite[Chapters~4~\&~5]{Romik} for more details; here, we simply state a result that we will employ later. Let $R_{k\times \ell}$ denote the $k\times \ell$ rectangle poset, which is simply the product of a chain of length $k$ and a chain of length $\ell$. The following result is originally due to Cohn--Elkies--Propp \cite{CohnElkiesPropp} and Jockusch--Propp--Shor \cite{JockuschProppShor} (using different language); it appears as \cite[Theorem~2.4]{Romik}.

\begin{theorem}[\cite{CohnElkiesPropp, JockuschProppShor}]\label{thm:rectangle}
Let $(k_n)_{n\geq 1}$ and $(\ell_n)_{n\geq 1}$ be sequences of positive integers such that the limit $(\overline{k},\overline{\ell})=\lim\limits_{n\to\infty}\frac{1}{n}(k_n,\ell_n)$ exists. As $n\to\infty$, we have \[\mathcal E(J(R_{k_n\times \ell_n}))=\frac{1}{p}\left(\overline{k}+\overline{\ell}+2\sqrt{(1-p)\overline k\,\overline\ell}\right)n+o(n).\]
\end{theorem}

The next theorem follows from \eqref{eq:distributive_max_chains} and a Chernoff bound; we provide the details in \Cref{sec:distributive}. 

\begin{theorem}\label{thm:distributive_Chernoff}
Let $(P_n)_{n\geq 1}$ be a sequence of posets. Suppose there exist constants $\Gamma$ and $\mu$ such that $\left\lvert\MC(P_n)\right\rvert\leq\Gamma^{(1+o(1))n}$ and such that the maximum size of a chain in $P_n$ is at most $\mu n$. As $n\to\infty$, we have \[\mathcal E(J(P_n))\leq\frac{1}{p}\left(\mu+\log\Gamma+\sqrt{2\mu\log\Gamma+(\log\Gamma)^2}\right)n+o(n).\] 
\end{theorem}

\subsection{Trim Lattices}
Thomas \cite{Thomas} introduced \emph{trim lattices} as generalizations of distributive lattices that need not be graded. There are several notable examples of trim lattices such as Cambrian lattices and $\nu$-Tamari lattices; see \cite[Section~7]{ThomasWilliams} for a more extensive list. In \cite{ThomasWilliams}, Thomas and Williams investigated fascinating dynamical properties of trim lattices and also related their results to quiver representation theory. They defined the \dfn{spine} of a trim lattice $L$ to be the set $\spine(L)$ of elements that lie on at least one maximum-length chain of $L$, and they showed that $\spine(L)$ is a distributive sublattice of $L$. The following is one of our main results. 

\begin{theorem}\label{thm:spine}
If $L$ is a trim lattice, then \[\mathcal E(L)\leq\mathcal E(\spine(L)).\]
\end{theorem}

In fact, we have the following more general result, which we will obtain as a corollary of \Cref{thm:spine}, \eqref{eq:distributive_max_chains}, and known properties of trim lattices. Note that quotients and intervals of trim lattices are trim. 

\begin{corollary}\label{cor:spine}
If $L$ is a trim lattice and $L'$ is a quotient or an interval of $L$, then \[\mathcal E(L')\leq\mathcal E(\spine(L)).\]
\end{corollary}

The proof of \Cref{thm:spine} is lattice-theoretic and relies on several known properties of trim lattices. Because $\spine(L)$ is distributive, we can interpret its Ungarian Markov chain as last-passage percolation and use the results discussed in \Cref{subsec:distributive} to obtain immediate upper bounds for $\mathcal E(\spine(L))$. Thus, \Cref{thm:spine,cor:spine} combine lattice theory and probability theory to yield a powerful method for bounding $\mathcal E(L)$ for any trim lattice $L$. We illustrate this method with some applications for some special classes of trim lattices, whose definitions we postpone. 

Associated to a Coxeter element $c$ of a finite Coxeter group $W$ is an important lattice $\Camb_c$, called the \emph{$c$-Cambrian lattice}, which Reading introduced in \cite{ReadingCambrian, ReadingSortableCambrian}. Cambrian lattices have now been studied thoroughly because of their rich combinatorial and lattice-theoretic properties and their connections to representation theory, cluster algebras, and polyhedral geometry \cite{EmiliesAndRalf, Chapoton, Ingalls, ReadingSpeyerFans, ReadingSpeyerFrameworks, WhyTheFuss, ThomasWilliams}. Because we are primarily interested in asymptotic results, it makes little sense for us to deal with Coxeter groups of exceptional types. Thus, we will focus on the classical types $A$, $B$, and $D$.  

\begin{theorem}\label{thm:CambA}
For each $n\geq 1$, choose a Coxeter element $c^{(n)}$ of the Coxeter group $A_n$. As $n\to\infty$, we have \[\mathcal E\left(\Camb_{c^{(n)}}\right)\leq \frac{1}{p}\left(2+2\sqrt{1-p}\right)n+o(n).\] 
\end{theorem}

The Coxeter graph of the Coxeter group $B_n$ is a path with vertices $s_0,\ldots,s_{n-1}$ and edges of the form $\{s_i,s_{i+1}\}$ for $0\leq i\leq n-2$. A Coxeter element of $B_n$ is uniquely determined by orienting the edges of this graph; let $r(c)$ denote the number of edges $\{s_i,s_{i+1}\}$ that are oriented from $s_i$ to $s_{i+1}$ in the orientation corresponding to a Coxeter element $c$. 

\begin{theorem}\label{thm:CambB}
For each $n\geq 2$, choose a Coxeter element $c^{(n)}$ of the Coxeter group $B_n$, and assume that the limit $\overline{r}=\lim\limits_{n\to\infty}\frac{1}{n}r(c^{(n)})$ exists. As $n\to\infty$, we have \[\mathcal E\left(\Camb_{c^{(n)}}\right)\leq \frac{1}{p}\left(3+2\sqrt{(1-p)(2-\overline r)(1+\overline{r})}\right)n+o(n).\] 
\end{theorem}

The Coxeter graph of the Coxeter group $D_n$ has vertices $s_0,\ldots,s_{n-1}$, edges of the form $\{s_i,s_{i+1}\}$ for $1\leq i\leq n-2$, and an additional edge $\{s_0,s_2\}$. A Coxeter element of $D_n$ is uniquely determined by orienting the edges of this graph. Let $r(c)$ denote the number of edges $\{s_i,s_{i+1}\}$ that are oriented from $s_i$ to $s_{i+1}$ in the orientation corresponding to a Coxeter element $c$ (note that $r(c)$ does not depend on the orientation of $\{s_0,s_2\}$), and let $u(c)$ denote the number of edges in a maximum-length directed path in this orientation.

\begin{theorem}\label{thm:CambD}
For each $n\geq 4$, choose a Coxeter element $c^{(n)}$ of the Coxeter group $D_n$, and assume that the limits $\overline{r}=\lim\limits_{n\to\infty}\frac{1}{n}r(c^{(n)})$ and $\overline{u}=\lim\limits_{n\to\infty}\frac{1}{n}u(c^{(n)})$ exist. As $n\to\infty$, we have \begin{align*}
\mathcal E\left(\Camb_{c^{(n)}}\right)&\leq \frac{1}{p}\left(6+4\sqrt{(1-p)(2-\overline r)(1+\overline r)}\right)n+o(n)
\end{align*}
and 
\begin{align*}
\mathcal E(\Camb_{c^{(n)}})\leq \frac{1}{p}\left(2+\overline{u}+\log\left(5\cdot 2^{\overline{u}}\right)+\sqrt{2(2+\overline{u})\log \left(5\cdot 2^{\overline{u}}\right)+\left(\log \left(5\cdot 2^{\overline{u}}\right)\right)^2}\right)n+o(n).
\end{align*}  
\end{theorem}

\begin{remark}
Let us consider numerical approximations of the upper bounds in \Cref{thm:CambA,thm:CambB,thm:CambD} when $p=1/2$. The bound in \Cref{thm:CambA} is $(6.82843+o(1))n$. The bound in \Cref{thm:CambB} ranges from $(10+o(1))n$ to $(10.24264+o(1))n$ as $\overline r$ ranges from $0$ to $1$. The first bound in \Cref{thm:CambD} ranges from $(20+o(1))n$ to $(20.48528+o(1))n$ as $\overline r$ ranges from $0$ to $1$, while the second bound ranges from $(13.22822+o(1))n$ to $(19.34986+o(1))n$ as $\overline u$ ranges from $0$ to $1$. (So the second bound in \Cref{thm:CambD} is always better than the first when $p=1/2$, but the first bound ends up being better when $p$ is closer to $1$.) 
\end{remark}

\begin{remark}
If we maximize the bounds in \Cref{thm:CambA,thm:CambB,thm:CambD} over all values of $\overline r$ and $\overline u$, we obtain the upper bounds \[\frac{1}{p}\left(2+2\sqrt{1-p}\right)n+o(n),\] \[\frac{1}{p}\left(3+3\sqrt{1-p}\right)n+o(n),\]
 and \[\frac{1}{p}\min\left\{6+6\sqrt{1-p},3+\log(10)+\sqrt{6\log(10)+(\log(10))^2}\right\}n+o(n),\] respectively, which are linear in $n$ and apply uniformly to all Cambrian lattices of the prescribed types. It is not clear how one could obtain such linear bounds (with any leading coefficients whatsoever) without the use of \Cref{thm:spine}. 
\end{remark}

For the sake of completeness, we will also prove the following very easy result about Cambrian lattices of dihedral groups. 

\begin{theorem}\label{thm:dihedral}
Let $c$ be a Coxeter element of $I_2(m)$, the dihedral group of order $2m$. Then \[\mathcal E(\Camb_c)=\frac{1+m(1-p)}{2p-p^2}.\]
\end{theorem}

Cambrian lattices are generalizations of Tamari lattices, which are significant objects in algebraic combinatorics that have been studied since the work of Tamari in 1962 \cite{Tamari}. Bergeron and Pr\'eville-Ratelle \cite{Bergeron} introduced different generalizations of Tamari lattices called \emph{$m$-Tamari lattices} in their study of trivariate diagonal harmonics. These lattices have now received a great deal of further attention \cite{BousquetRep, BousquetIntervals, Chatel, DefantLin}. Going further, Pr\'eville-Ratelle and Viennot \cite{PrevilleViennot} generalized $m$-Tamari lattices by defining the \emph{$\nu$-Tamari lattice} $\Tam(\nu)$, where $\nu$ is an arbitrary lattice path consisting of unit north and east steps; this more general class of lattices has been connected to diagonal coinvariant spaces \cite{PrevilleViennot}, polyhedral geometry \cite{vonBellSchroder, vonBellUnifying, CeballosGeometry}, and combinatorial dynamics \cite{DefantMeeting, DefantLin}. 

We will provide a simple explicit description of the spine of $\Tam(\nu)$ (see \Cref{prop:Tam_nu_Cells}), which appears to be new. Combining this description with \Cref{thm:rectangle}, we will deduce the following upper bound for $\mathcal E(\Tam(\nu))$. 

\begin{theorem}\label{thm:nu-Tamari}
For each $\ell\geq 1$, choose a lattice path $\nu^{(\ell)}$ consisting of a total of $\ell$ north and east steps. Let $n_\ell$ be the number of north steps in $\nu^{(\ell)}$, and assume that the limit $\overline n=\lim\limits_{\ell\to\infty}\frac{1}{\ell}n_\ell$ exists. As $\ell\to\infty$, we have \[\mathcal E(\Tam(\nu^{(\ell)}))\leq\frac{1}{p}\left(1+2\sqrt{(1-p)\overline n(1-\overline n)}\right)\ell+o(\ell).\]
\end{theorem}

\subsection{Tamari Lattices}\label{subsec:Tamari}

The original Tamari lattice $\Tam_n$ is a Cambrian lattice of type $A_{n-1}$ and is also $\Tam(\nu)$ when $\nu=(\text{N}\text{E})^n$. By specializing either \Cref{thm:CambA} or \Cref{thm:nu-Tamari}, we find that $\mathcal E(\Tam(\nu))\leq \frac{1}{p}(2+2\sqrt{1-p})n+o(n)$. Because Tamari lattices are so fundamental, we will analyze them on their own and derive a better upper bound. 

In \cite{Bruss}, Bruss and O'Cinneide studied the asymptotic behavior of $\rho_p(n)$, which is defined to be the probability that the maximum of $n$ independent geometric random variables, each with expected value $1/p$, is attained uniquely. Somewhat surprisingly, they found that $\lim\limits_{n\to\infty}\rho_p(n)$ does not exist if $p<1$ and that $\lim\limits_{n\to\infty}(\rho_p(n)-\Upsilon_p(n))=0$, where \[\Upsilon_p(x)=\begin{cases} px\displaystyle\sum_{k\in\mathbb Z}(1-p)^ke^{-(1-p)^kx} & \mbox{ if } p<1; \\
0 & \mbox{ if }p=1. \end{cases}\] Note that $\Upsilon_p((1-p)x)=\Upsilon_p(x)$ for all $x>0$.
Therefore, the quantity \[\overline \rho_p=\max_{0<x<1}\Upsilon_p(x)\] is equal to $\limsup\limits_{n\to\infty}\rho_p(n)$. When $p=1/2$, we have $\overline\rho_{1/2}\approx 0.72136$. 

\begin{theorem}\label{thm:Tamari}
We have \[\mathcal E(\Tam_n)\leq \frac{2}{p}\left(\sqrt{\overline\rho_p(1+\overline\rho_p)}-\overline\rho_p\right)n+o(n)\] as $n\to\infty$.
\end{theorem}

\begin{remark}
When $p=1/2$, the upper bound for $\mathcal E(\Tam_n)$ provided by either \Cref{thm:CambA} or \Cref{thm:nu-Tamari} is $(4+2\sqrt{2}+o(1))n\approx(6.82843+o(1))n$, while the upper bound in \Cref{thm:Tamari} is $4\left(\sqrt{\overline\rho_{1/2}(1+\overline\rho_{1/2})}-\overline\rho_{1/2}+o(1)\right)n\approx(1.57186+o(1))n$. 
\end{remark}

\subsection{Outline} 
In \Cref{sec:Preliminaries}, we recall some basic definitions concerning posets and lattices, and we prove a lemma that bounds the tails of a sum of i.i.d.\ geometric random variables. \Cref{sec:Weak} is devoted to proving \Cref{thm:weak}, which bounds $\mathcal E(S_n)$. In \Cref{sec:distributive}, we consider $\mathcal E(L)$ when $L$ is distributive; we state some corollaries that follow from interpreting $\bU_L$ as last-passage percolation, and we prove \Cref{thm:distributive_Chernoff}. \Cref{sec:Trim} recalls necessary background about trim lattices and proves \Cref{thm:spine}. In \Cref{sec:Cambrian}, we recall background about Cambrian lattices, describe the spines of Cambrian lattices, and combine \Cref{thm:spine} with \Cref{thm:rectangle,thm:distributive_Chernoff} to prove \Cref{thm:CambA,thm:CambB,thm:CambD}; we also quickly prove \Cref{thm:dihedral}. In \Cref{sec:nu-Tamari}, we discuss background about $\nu$-Tamari lattices, derive a simple description of the spines of $\nu$-Tamari lattices, and use this description to deduce \Cref{thm:nu-Tamari} from \Cref{thm:rectangle,thm:spine}. \Cref{sec:Weak} provides a direct combinatorial proof of \Cref{thm:Tamari}, which bounds $\mathcal E(\Tam_n)$ from above. Finally, \Cref{sec:Future} lists several enticing suggestions for further research. 

\section{Preliminaries}\label{sec:Preliminaries} 

We assume basic familiarity with the theory of posets (partially ordered sets), as discussed in \cite[Chapter~3]{Stanley}. As mentioned in the introduction, all posets considered in this article are assumed to be finite. Here, we recall some notions that we will need. 

Let $P$ be a poset. A \dfn{subposet} of $P$ is a poset $P'\subseteq P$ such that if $x,y\in P'$ satisfy $x\leq y$ in $P'$, then $x\leq y$ in $P$. The \dfn{dual} of $P$ is the poset $P^*$ with the same underlying set as $P$ such that $x\leq y$ in $P$ if and only if $y\leq x$ in $P^*$. Given $x,y\in P$ with $x\leq y$, the \dfn{interval} from $x$ to $y$ is the set $[x,y]=\{z\in P:x\leq z\leq y\}$. If $x<y$ and $[x,y]=\{x,y\}$, then we say $y$ \dfn{covers} $x$ and write $x\lessdot y$. We write $\cov_P(y)=\{x\in P:x\lessdot y\}$ for the set of elements of $P$ that are covered by $y$. A \dfn{chain} of $P$ is a totally-ordered subset of $P$; we say a chain is \dfn{maximal} if it is not properly contained in another chain. Let $\MC(P)$ denote the set of maximal chains of $P$. The \dfn{length} of a chain $\mathcal C$ of $P$ is $|\mathcal C|-1$. The \dfn{length} of $P$ is the maximum of the lengths of the chains of $P$. We say $P$ is \dfn{graded} if all of its maximal chains have the same length. 

As discussed in \Cref{sec:intro}, a \dfn{lattice} is a poset $L$ such that any two elements $x,y\in L$ have a meet $x\wedge y$ and a join $x\vee y$. Write $\bigwedge X$ and $\bigvee X$ for the meet and join, respectively, of a nonempty set $X\subseteq L$. The lattice $L$ has a unique minimal element $\hat 0$ and a unique maximal element $\hat 1$; we use the convention that $\bigwedge \emptyset=\hat 1$ and $\bigvee\emptyset=\hat 0$. A \dfn{sublattice} of $L$ is a lattice $L'\subseteq L$ whose meet and join operations agree with those of $L$. We implicitly view intervals of $L$ as sublattices of $L$.   

Given lattices $L$ and $L'$, a map $\varphi\colon L\to L'$ is a \dfn{lattice homomorphism} if $\varphi(x\wedge y)=\varphi(x)\wedge\varphi(y)$ and $\varphi(x\vee y)=\varphi(x)\vee\varphi(y)$ for all $x,y\in L$. We say $L'$ is a \dfn{quotient} of $L$ if there exists a surjective lattice homomorphism from $L$ to $L'$. 

Throughout this article, we omit floor and ceiling symbols when they do not affect the relevant asymptotics.

We will frequently need the following lemma, which allows us to bound the upper tails of a sum of i.i.d.\ geometric random variables. 

\begin{lemma}\label{lem:geometric_tail}
Let $G_1,\ldots,G_k$ be independent geometric random variables with parameter $p$. For each real number $\gamma\geq 1$, we have \[\P\left(\sum_{i=1}^kG_i>\gamma k/p\right)\leq e^{-\frac{\gamma k}{2}\left(1-\gamma^{-1}\right)^2}.\] 
\end{lemma}

\begin{proof}
Extend $G_1,\dots,G_k$ to an infinite sequence $G_1,G_2,\dots$ of independent geometric random variables with parameter $p$.
Let
\[ Y_m=\begin{cases} 1 & \mbox{ if } m=\sum_{i=1}^jG_i\text{ for some }j\geq 1; \\
0 & \mbox{ otherwise}. \end{cases} \]
Then $Y_1,Y_2,\ldots$ is a sequence of independent Bernoulli random variables, each with expected value $p$. Thus, $\sum_{1\leq m\leq \gamma k/p}Y_m$ is a binomial random variable with parameters $\gamma k/p$ and $p$. Let $\mu=\gamma k$ and $\delta=1-\gamma^{-1}$ so that $k=(1-\delta)\mu$. A standard Chernoff bound tells us that
\[\P\left(\sum_{i=1}^kG_i>\gamma k/p\right)=\P\left(\sum_{1\leq m \leq \gamma k/p} Y_m < k\right) \leq e^{-\mu\delta^2/2}=e^{-\frac{\gamma k}{2}\left(1-\gamma^{-1}\right)^2}. \qedhere \]
\end{proof}

\section{Symmetric Groups Under the Weak Order}\label{sec:Weak}

Let $s_i$ denote the transposition $(i\,\,i+1)$ in the symmetric group $S_n$. The \dfn{weak order} on $S_n$ is the partial order in which there is a cover relation $w\lessdot w'$ whenever $i$ is a descent of $w'$ and $w'=ws_i$ (more precisely, this is the \emph{right} weak order). It is well known that the weak order on $S_n$ is a lattice; we will simply write $S_n$ for this lattice. 

In this section, we analyze the Ungarian Markov chains $\bU_{S_n}$. We will prove \Cref{thm:weak}, which provides asymptotic estimates for the expected value $\mathcal E(S_n)$. We assume throughout this section that our fixed probability $p$ is strictly between $0$ and $1$. 

To begin, let us derive the lower bound in \Cref{thm:weak} from Ungar's \Cref{thm:Ungar}. 

\begin{lemma}\label{lem:lower_weak}
We have $n-1+o(1)\leq\mathcal E(S_n)$. 
\end{lemma}

\begin{proof}
Let us run $\bU_{S_n}$ starting at the top element $\hat 1=n(n-1)\cdots 1$. This element $\hat 1$ covers $n-1$ elements of $S_n$, so $\bU_{S_n}$ has the transition probability $\P(\hat 1\to\hat 1)=(1-p)^{n-1}$. The only Ungar move that we can apply to $\hat 1$ to obtain $\hat 0$ is the maximal Ungar move, and this has transition probability $\P(\hat 1\to\hat 0)=p^{n-1}$. If the first random Ungar move that we apply is not trivial or maximal, then it follows immediately from \Cref{thm:Ungar} that the total number of steps we will need to reach $\hat 0$ is at least $n-1$. Therefore, 
\begin{align*}
\mathcal E(S_n)&\geq (1-p)^{n-1}(1+\mathcal E(S_n))+p^{n-1}\cdot 1+\sum_{\emptyset\neq T\subsetneq[n-1]}p^{|T|}(1-p)^{n-1-|T|}(n-1) \\ 
&=(1-p)^{n-1}(1+\mathcal E(S_n))+p^{n-1}+(1-(1-p)^{n-1}-p^{n-1})(n-1) \\ 
&=1+(1-p)^{n-1}\mathcal E(S_n)+(1-(1-p)^{n-1}-p^{n-1})(n-2)
\end{align*}
It follows that \[\mathcal E(S_n)\geq\frac{1+(1-(1-p)^{n-1}-p^{n-1})(n-2)}{1-(1-p)^{n-1}} =n-2+\frac{1-p^{n-1}(n-2)}{1-(1-p)^{n-1}} =n-1+o(1). \qedhere\] 
\end{proof} 

We now proceed to prove the upper bound in \Cref{thm:weak}. 

Given a permutation $w\in S_n$, we write $\DB(w)$ for the set $\{w(i+1):i\in\Des(w)\}$ of \dfn{descent bottoms} of $w$. For integers $\beta\in[n-1]$ and $j\geq 0$, let $X_j^{(\beta)}$ be a Bernoulli random variable with expected value $p$; assume that all of these random variables for different choices of $\beta$ and $j$ are independent. We can simulate the Markov chain $\bU_{S_n}$ using these random variables as follows. Starting with the decreasing permutation $\sigma_0=n(n-1)\cdots 1$, we are going to create a sequence $\sigma_0,\sigma_1,\ldots$ of permutations. Suppose we have already generated the permutations $\sigma_0,\ldots,\sigma_t$. We will define $\sigma_{t+1}$ to be the permutation obtained by choosing a random subset $T_t$ of $\Des(\sigma_t)$ and then applying the corresponding random Ungar move to $\sigma_t$. In Coxeter-theoretic terminology, this means that $\sigma_{t+1}=\sigma_tw_0(T_t)$, where $w_0(T_t)$ is the maximal element of the parabolic subgroup of $S_n$ generated by $\{s_i:i\in T_t\}$. For each $\beta\in\DB(\sigma_t)$, let $j_t(\beta)=|\{\ell\in\{0,\ldots,t-1\}:\beta\in\DB(\sigma_\ell)\}|$. The random variables from above now enter the picture: we define \[T_t=\left\{i\in\Des(\sigma_t):X_{j_t(\sigma_t(i+1))}^{(\sigma_t(i+1))}=1\right\}.\] Note that for permutations $w,w'\in S_n$, the conditional probability $\mathbb P(\sigma_{t+1}=w'\mid \sigma_t=w)$ is equal to the transition probability $\mathbb P(w\to w')$ in $\bU_{S_n}$. 

This construction is designed so that if $\beta$ appears immediately to the right of some entry $\alpha>\beta$ in $\sigma_t$, then $\beta$ will move to the left of $\alpha$ when we transition to $\sigma_{t+1}$ if and only if $X_{j_t(\beta)}^{(\beta)}=1$. Our definition of $j_t(\beta)$ is such that $X_0^{(\beta)},\ldots,X_{j_t(\beta)-1}^{(\beta)}$ are precisely the random variables from the sequence $(X_j^{(\beta)})_{j\geq 0}$ that were already used when we generated the permutations $\sigma_0,\ldots,\sigma_t$ (so $X_{j_t(\beta)}^{(\beta)}$ is the next available random variable from this sequence). 

Let $N$ be the unique integer such that $\sigma_{N-1}\neq\sigma_N=12\cdots n$. Then $\sigma_t=12\cdots n$ for all $t\geq N$. For $\beta\in[n-1]$, let $\eta_0^{(\beta)}<\cdots<\eta_{q(\beta)}^{(\beta)}$ be the indices $t$ such that $\beta\in\DB(\sigma_t)$. If $\beta<\beta'$ and $\beta$ appears to the left of $\beta'$ in a permutation $\sigma_t$, then $\beta$ must also appear to the left of $\beta'$ in all of the permutations $\sigma_{t+1},\sigma_{t+2},\ldots$. This implies that 
\begin{equation}\label{eq:3}
\sum_{j=0}^{q(\beta)}X_j^{(\beta)}\leq n-\beta.
\end{equation}
Observe that $N-1=\max\limits_{1\leq \beta\leq n-1}\eta_{q(\beta)}^{(\beta)}$. 

\begin{lemma}\label{lem:technical_weak}
Assume that \[\sum_{j=a}^{b-1}X_j^{(\beta)}\geq p(b-a)-\sqrt{8p\log n}\cdot\sqrt{b-a}\] for all $1\leq \beta\leq n-1$ and $0\leq a<b\leq q(\beta)+1$. Then $N\leq \frac{8}{p}n\log n+O(n)$.  
\end{lemma}

\begin{proof}
To ease notation, let $\chi(n)=\sqrt{8p\log n}$. Let $r_\beta(s)=n-\sigma_s^{-1}(\beta)$ be the number of entries that appear to the right of $\beta$ in $\sigma_s$. We will prove by induction on $\beta$ that 
\begin{equation}\label{eq:induction_i}
r_\beta(s)\geq \min\{n-\beta,ps-\chi(n)\sqrt{\beta s}-(p+1)(\beta-1)\}
\end{equation}
for all $1\leq \beta\leq n$ and $0\leq s\leq N$. 

We first consider the base case when $\beta=1$. Note that $\eta_j^{(1)}=j$ for every $0\leq j\leq q(1)$.
Suppose $r_1(s)<n-1$. Then $s\leq q(1)$, so \[r_1(s)\geq\sum_{j=0}^{s-1}X_j^{(1)}\geq ps-\chi(n)\sqrt{s},\] where we obtained the last inequality by setting $a=0$ and $b=s$ in the hypothesis of the lemma. This proves the base case. 

We may now assume $\beta\geq 2$ and proceed by induction on $\beta$. Suppose by way of contradiction that there exists $0\leq s\leq N$ such that $r_\beta(s)<n-\beta$ and $r_\beta(s)<ps-\chi(n)\sqrt{\beta s}-(p+1)(\beta-1)$. Then it follows from our induction hypothesis that $r_\beta(s)\leq\min\{r_1(s),\ldots,r_{\beta-1}(s)\}-2$. Let $s^*$ be the largest integer such that $s^*\leq s-1$ and $r_\beta(s^*)=\min\{r_1(s^*),\ldots,r_{\beta-1}(s^*)\}-1$ (so $\beta$ appears immediately to the right of the rightmost entry from the list $1,\ldots,\beta-1$ in $\sigma_{s^*}$); it is straightforward to see from the original definition of an Ungar move (for $S_n$) that such an integer $s^*$ must exist.
Because the entry immediately to the left of $\beta$ in $\sigma_{s^*}$ is smaller than $\beta$, we must have $r_{\beta}(s^*+1)=r_{\beta}(s^*)=\min\{r_1(s^*),\ldots,r_{\beta-1}(s^*)\}-1$. It follows from our choice of $s^*$ that $\beta$ is a descent bottom of each of the permutations $\sigma_{s^*+1},\ldots,\sigma_{s-1}$. Consequently, there exists $k$ such that the numbers $\eta_k^{(\beta)},\ldots,\eta_{k+s-s^*-2}^{(\beta)}$ are the same as the numbers $s^*+1,\ldots,s-1$, respectively. Applying the hypothesis of the lemma with $a=k$ and $b=k+s-s^*-1$, we find that 
\begin{align}\label{eq:1}
r_\beta(s)&\geq r_\beta(s^*+1)+\sum_{j=k}^{k+s-s^*-2}X_j^{(\beta)} \nonumber\\ &= \min\{r_1(s^*),\ldots,r_{\beta-1}(s^*)\}-1+\sum_{j=k}^{k+s-s^*-2}X_j^{(\beta)} \nonumber\\ &\geq  \min\{r_1(s^*),\ldots,r_{\beta-1}(s^*)\}-1+p(s-s^*-1)-\chi(n)\sqrt{s-s^*-1}.
\end{align}
Our induction hypothesis tells us that \[\min\{r_1(s^*),\ldots,r_{\beta-1}(s^*)\}\geq\min\{n-\beta+1,ps^*-\chi(n)\sqrt{(\beta-1)s^*}-(p+1)(\beta-2)\}.\] If $\min\{r_1(s^*),\ldots,r_{\beta-1}(s^*)\}=n-\beta+1$, then $r_\beta(s^*)=n-\beta$, contradicting the fact that $r_\beta(s)\leq\min\{r_1(s),\ldots,r_{\beta-1}(s)\}-2$ (if $\beta$ appears in position $\beta$ in $\sigma_{s^*}$, then it cannot appear further to the right in $\sigma_s$). This shows that $\min\{r_1(s^*),\ldots,r_{\beta-1}(s^*)\}<n-\beta+1$, so $\min\{r_1(s^*),\ldots,r_{\beta-1}(s^*)\}\geq ps^*-\chi(n)\sqrt{(\beta-1)s^*}-(p+1)(\beta-2)$. Combining this with \eqref{eq:1} yields 
\begin{align*}
r_\beta(s)&\geq ps^*-\chi(n)\sqrt{(\beta-1)s^*}-(p+1)(\beta-2)-1+p(s-s^*-1)-\chi(n)\sqrt{s-s^*-1} \\ 
&=ps-\chi(n)\left[\sqrt{(\beta-1)s^*}+\sqrt{s-s^*-1}\right]-(p+1)(\beta-1) \\ 
&\geq ps-\chi(n)\sqrt{\beta s}-(p+1)(\beta-1), 
\end{align*}
where the inequality $\sqrt{(\beta-1)s^*}+\sqrt{s-s^*-1}\leq\sqrt{\beta s}$ used to deduce the last line follows from basic calculus. This contradicts our original choice of $s$, so our inductive proof of \eqref{eq:induction_i} is complete. 

For every $0\leq s\leq N-1$, there exists $\beta\in[n]$ such that $r_\beta(s)<n-\beta$. Therefore, if $k$ is an integer such that 
\begin{equation}\label{eq:2}
pk-\chi(n)\sqrt{\beta k}-(p+1)(\beta-1)\geq n-\beta
\end{equation} for all $1\leq \beta\leq n$, then \eqref{eq:induction_i} tells us that $N\leq k$. If $k$ satisfies the inequality \eqref{eq:2} when $\beta=n$, then it also satisfies \eqref{eq:2} for all $\beta\in[n]$.
Thus, to get an upper bound for $N$, we simply need to find an integer $k$ such that $pk-\chi(n)\sqrt{n}\sqrt{k}-(p+1)(n-1)\geq 0$. The left-hand side of this inequality is quadratic in $\sqrt{k}$, so we can use the quadratic formula (and the definition of $\chi(n)$) to find that we just need $k$ to satisfy \[\sqrt{k}\geq\frac{1}{2p}\left[\sqrt{8pn\log n}+\sqrt{8pn\log n+4p(p+1)(n-1)}\right].\] 
Now, $\sqrt{8pn\log n+4p(p+1)(n-1)}=\sqrt{8pn\log n}+O\left(\sqrt{n/\log n}\right)$, so \[N\leq \left(\frac{1}{2p}\left(2\sqrt{8pn\log n}+O\left(\sqrt{n/\log n}\right)\right)\right)^2=\frac{8}{p}n\log n+O(n). \qedhere\] 
\end{proof}

\begin{lemma}\label{lem:q(i)}
We have \[\P\left(\max_{1\leq \beta\leq n-1}q(\beta)>2n/p\right)<ne^{-n/4}.\]
\end{lemma}

\begin{proof}
We know by \eqref{eq:3} that $\sum_{j=0}^{q(\beta)}X_j^{(\beta)}<n$. This implies that $q(\beta)$ is bounded above by a sum of $n$ independent geometric random variables, each with expected value $1/p$. Setting $k=n$ and $\gamma=2$ in \Cref{lem:geometric_tail}, we find that \[\mathbb P(q(\beta)>2n/p)\leq e^{-n/4}.\] 
The desired result now follows from taking a union bound over all $1\leq \beta\leq n-1$. 
\end{proof}

We can now finish the proof of the main result of this section. 

\begin{proof}[Proof of \Cref{thm:weak}]
The lower bound in \Cref{thm:weak} is \Cref{lem:lower_weak}, so we just need to prove the upper bound. 

Consider simulating $\bU_{S_n}$ by creating the sequence of permutations (i.e., states) $\sigma_0,\ldots,\sigma_N$ (where $\sigma_0=n(n-1)\cdots 1$ and $\sigma_N=12\cdots n\neq\sigma_{N-1}$) as above. Preserve the notation from earlier. Every maximal chain in the weak order on $S_n$ has $\binom{n}{2}+1<n^2/2$ elements. At each step during the Markov chain before we reach the bottom state $12\cdots n$, the probability of moving to a strictly lower state is at least $p$. Therefore, $N\leq\sum_{\ell=1}^{\binom{n}{2}}G_\ell$, where $G_1,\ldots,G_{\binom{n}{2}}$ are independent geometric random variables with parameter $p$. Appealing to \Cref{lem:geometric_tail}, we find that
\begin{align}\label{eq:eventA}\sum_{m=\left\lceil n^2/p\right\rceil+1}^{\infty}m\P(N=m)&\leq \sum_{m=\left\lceil n^2/p\right\rceil}^{\infty}(m+1)\P(N>m) \nonumber\\ 
&\leq \sum_{m=\left\lceil n^2/p\right\rceil}^{\infty}(m+1)\P\left(\sum_{\ell=1}^{\binom{n}{2}}G_\ell>m\right) \nonumber\\
&\leq \sum_{m=\left\lceil n^2/p\right\rceil}^{\infty}(m+1)\exp\left(-\frac{pm}{2}\left(1-\frac{1}{pm}\binom{n}{2}\right)^2\right) \nonumber\\ 
&\leq \sum_{m=\left\lceil n^2/p\right\rceil}^{\infty}(m+1)\exp\left(-pm/8\right) \nonumber\\ 
&=O\left(n^2e^{-n^2/8}\right).
\end{align}

Let $A$ be the event that $\max\limits_{1\leq \beta\leq n-1}q(\beta)\leq 2n/p$, and let $A'$ be the event that \[\sum_{j=a}^{b-1}X_j^{(\beta)}\geq p(b-a)-\sqrt{8p\log n}\cdot\sqrt{b-a}\] for all $1\leq \beta\leq n-1$ and $0\leq a<b\leq q(\beta)+1$. For any particular $1\leq \beta\leq n-1$ and $0\leq a<b\leq q(\beta)+1$, we can use a Chernoff bound to see that \[\P\left(\sum_{j=a}^{b-1}X_j^{(\beta)}<p(b-a)-\sqrt{8p\log n}\cdot\sqrt{b-a}\right)\leq n^{-4}.\] Therefore, \[\P(A\setminus A')\leq \sum_{1\leq \beta\leq n-1}\sum_{0\leq a<b\leq 2n/p+1}n^{-4}=O(n^{-1}).\] Combined with \Cref{lem:q(i)}, this shows that $\P(\neg(A\cap A'))=O(n^{-1})$. According to \Cref{lem:technical_weak}, we have $N\leq \frac{8}{p}n\log n+O(n)$ if $A\cap A'$ occurs. Consequently, 
\begin{align*}
\sum_{m=1}^{\left\lceil n^2/p\right\rceil}m\P(N=m)&=\sum_{m=1}^{\left\lceil n^2/p\right\rceil}m\P(A\cap A'\cap(N=m))+\sum_{m=1}^{\left\lceil n^2/p\right\rceil}m\P(\neg(A\cap A')\cap(N=m)) \\ 
&\leq\mathbb E(N\mid(A\cap A'))+\left\lceil n^2/p\right\rceil\P(\neg(A\cap A')) \\ 
&\leq \frac{8}{p}n\log n+O(n).
\end{align*}
Combining this with \eqref{eq:eventA} shows that $\mathcal E(S_n)=\sum_{m=1}^\infty m\P(N=m)=\frac{8}{p}n\log n+O(n)$, as desired. 
\end{proof}

\section{Distributive Lattices}\label{sec:distributive}
Because Ungarian Markov chains on distributive lattices can be reformulated in terms of last-passage percolation with geometric weights---a well-studied topic in probability---we do not have too many new things to say about them. This short section is devoted to expounding upon some of the discussion from \Cref{subsec:distributive}. 

We first mention two immediate corollaries of the equation \eqref{eq:distributive_max_chains} that are not obvious from the original definition of Ungarian Markov chains. 

\begin{corollary}\label{cor:subposet}
If $P'$ is a subposet of a poset $P$, then $\mathcal E(J(P'))\leq\mathcal E(J(P))$. 
\end{corollary}

\begin{proof}
This follows from \eqref{eq:distributive_max_chains} because every maximal chain of $P'$ is contained in a maximal chain of~$P$. 
\end{proof}
\begin{corollary}\label{cor:dual}
Let $L$ be a distributive lattice, and let $L^*$ be the dual of $L$. Then $\mathcal E(L^*)=\mathcal E(L)$. 
\end{corollary}
\begin{proof}
There is a poset $P$ such that $L\cong J(P)$ and $L^*\cong J(P^*)$. The result follows from \eqref{eq:distributive_max_chains} because every maximal chain of $P$ is equal (as a set) to a maximal chain of $P^*$ and vice versa. 
\end{proof}

\begin{remark}
The hypothesis that $L$ is distributive in \Cref{cor:dual} is crucial. For example, suppose $L$ is the $c$-Cambrian lattice of type $A_3$, where $c=s_1s_3s_2$ (see \Cref{subsec:CambrianA} for definitions). Then $L$ is both trim and semidistributive, but it is not distributive. We have $\mathcal E(L)\neq\mathcal E(L^*)$ when the probability $p$ is generic. 
\end{remark}

We now prove \Cref{thm:distributive_Chernoff}.  

\begin{proof}[Proof of \Cref{thm:distributive_Chernoff}]
As in \Cref{subsec:distributive}, consider a collection $(G_x)_{x\in P_n}$ of independent geometric random variables with parameter $p$. Let \[\delta=\frac{1}{p}\left(\mu+\log\Gamma+\sqrt{2\mu\log\Gamma+(\log\Gamma)^2}\right).\] Fix $\varepsilon>0$, 
and let $\eta=\frac{(\delta+\varepsilon) p}{\mu}$. One can check that this choice of $\eta$ guarantees $\frac{\eta\mu n}{2}\left(1-\eta^{-1}\right)^2>\log\Gamma$, so
\[\Gamma^{(1+o(1))n}e^{-\frac{\eta\mu n}{2}\left(1-\eta^{-1}\right)^2}=o(1).\]
Suppose $\mathcal C$ is a maximal chain of $P_n$ of size $k$. For each integer $m\geq \eta\mu n/p$, we can use \Cref{lem:geometric_tail} and the fact that $k\leq\mu n$ to compute 
\begin{align*}
\P\left(\sum_{x\in\mathcal C}G_x>m\right)\leq e^{-\frac{pm}{2}\left(1-\frac{k}{pm}\right)^2} \leq e^{-\frac{pm}{2}\left(1-\frac{\mu n}{pm}\right)^2} \leq e^{-\frac{pm}{2}\left(1-\eta^{-1}\right)^2}.
\end{align*}
Consequently, 
\begin{align*}
\sum_{m\geq\eta\mu n/p}\P\left(\sum_{x\in\mathcal C}G_x>m\right)\leq\sum_{m\geq\eta\mu n/p}e^{-\frac{pm}{2}\left(1-\eta^{-1}\right)^2}  
=O\left(e^{-\frac{p(\eta\mu n/p)}{2}\left(1-\eta^{-1}\right)^2}\right)  
=O\left(e^{-\frac{\eta\mu n}{2}\left(1-\eta^{-1}\right)^2}\right).
\end{align*}
By a union bound, this implies that 
\begin{align*}
\sum_{m\geq\eta\mu n/p}\P\left(\max_{\mathcal C\in\MC(P_n)}\sum_{x\in\mathcal C}G_x>m\right)&=O\left(\left\lvert\MC(P_n)\right\rvert e^{-\frac{\eta\mu n}{2}\left(1-\eta^{-1}\right)^2}\right) 
 \\ 
 &=O\left(\Gamma^{(1+o(1))n} e^{-\frac{\eta\mu n}{2}\left(1-\eta^{-1}\right)^2}\right) 
 \\ 
 &=o(1). 
\end{align*}
Finally, using \eqref{eq:distributive_max_chains}, we find that
\begin{align*}
\mathcal E(J(P_n))&=\mathbb E\left(\max_{\mathcal C\in\MC(P_n)}\sum_{x\in \mathcal C}G_x\right) \\ 
&=\sum_{m\geq 0}\P\left(\max_{\mathcal C\in\MC(P_n)}\sum_{x\in \mathcal C}G_x>m\right) \\ 
&=\sum_{0\leq m<\eta\mu n/p}\P\left(\max_{\mathcal C\in\MC(P_n)}\sum_{x\in \mathcal C}G_x>m\right)+\sum_{m\geq\eta\mu n/p}\P\left(\max_{\mathcal C\in\MC(P_n)}\sum_{x\in\mathcal C}G_x>m\right) \\ 
&\leq\frac{\eta\mu n}{p}+1+o(1).
\end{align*}
Since $\frac{\eta\mu n}{p}=(\delta+\varepsilon)n$ and $\varepsilon$ was arbitrary, this proves that $\mathcal E(J(P_n))\leq\delta n+o(n)$, as desired. 
\end{proof}

\section{Trim Lattices}\label{sec:Trim}

\subsection{Basics of Trim Lattices}\label{subsec:trim_basics}
Let $L$ be a lattice. An element $j\in L$ is called \dfn{join-irreducible} if it covers exactly one element of $L$; in this case, we write $j_*$ for the unique element of $L$ covered by $j$.
Dually, an element $m\in L$ is called \dfn{meet-irreducible} if it is covered by exactly one element of $L$; in this case, we write $m^*$ for the unique element of $L$ that covers $m$. Let $\mathcal J_L$ and $\mathcal M_L$ be the set of join-irreducible elements of $L$ and the set of meet-irreducible elements of $L$, respectively.
The length of $L$ is at most $|\mathcal J_L|$ and also at most $|\mathcal M_L|$; we say $L$ is \dfn{extremal} if its length is equal to both $|\mathcal J_L|$ and $|\mathcal M_L|$. 

An element $x$ of a lattice $L$ is called \dfn{left modular} if for all $y,z\in L$ with $y\leq z$, we have \[(y\vee x)\wedge z=y\vee(x\wedge z).\] We say $L$ is \dfn{left modular} if it has a maximal chain whose elements are all left modular. 

A lattice is \dfn{trim} if it is both extremal and left modular. Let us recall some facts about trim lattices from \cite{Semidistrim,Thomas,ThomasWilliams}.

Suppose $L$ is trim. For each join-irreducible element $j\in\mathcal J_L$, there is a unique meet-irreducible element $\kappa_L(j)\in\mathcal M_L$ such that \begin{equation}\label{eq:kappa_definition}
j\wedge\kappa_L(j)=j_*\quad\text{and}\quad j\vee\kappa_L(j)=(\kappa_L(j))^*.
\end{equation}
The resulting map $\kappa_L\colon\mathcal J_L\to\mathcal M_L$ is a bijection. The \dfn{Galois graph} of $L$ is the loopless directed graph $\Gal(L)$ with vertex set $\mathcal J_L$ such that for all distinct $j,j'\in\mathcal J_L$, there is an arrow $j\to j'$ in $\Gal(L)$ if and only if $j\not\leq\kappa_L(j')$. This graph is acyclic (i.e., it has no directed cycles), so we can define a partial order $\preceq$ on $\mathcal J_L$ by declaring that $j\preceq j'$ if there exists a directed path in $\Gal(L)$ from $j'$ to $j$. We call the resulting poset $\Pos(L)=(\mathcal J_L,\preceq)$ the \dfn{Galois poset} of $L$. 

\begin{proposition}[{\cite[Theorem~1]{Thomas}, \cite[Proposition~3.13]{ThomasWilliams}, \cite[Theorem~6.2~\&~Corollary~7.10]{Semidistrim}}]\label{prop:trim_interval}
Let $L$ be a trim lattice. If $u,v\in L$ are such that $u\leq v$, then the interval $[u,v]$ of $L$ is also a trim lattice. Moreover, $\Gal([u,v])$ is isomorphic to the subgraph of $\Gal(L)$ induced by $\{j\in\mathcal J_L:j\leq v, \kappa_L(j)\geq u\}$. 
\end{proposition}

\subsection{Spines}

Not all trim lattices are graded; in fact, a trim lattice is graded if and only if it is distributive \cite[Theorem~2]{Thomas}.
The \dfn{spine} of a trim lattice $L$, denoted $\spine(L)$, is the set of elements that belong to a maximum-length chain of $L$. According to \cite[Proposition~2.6]{ThomasWilliams}, the spine of $L$ is a distributive sublattice of $L$, and \begin{equation}\label{eq:spine_isomorphism}
\spine(L)\cong J(\Pos(L)).
\end{equation}

In this subsection, we will prove \Cref{thm:spine}, which states that $\mathcal E(L)\leq\mathcal E(\spine(L))$. First, we need the following lemma. 

\begin{lemma}\label{lem:subspine}
Let $L$ be a trim lattice, and let $x\in\spine(L)$. The interval $[\hat 0,x]$ of $L$ is a trim lattice whose spine is $\spine(L)\cap[\hat 0,x]$.  
\end{lemma}

\begin{proof}
We know by \Cref{prop:trim_interval} that $[\hat 0,x]$ is trim. Suppose $y\in\spine(L)\cap [\hat 0,x]$. Then $y\leq x$ in $\spine(L)$, so there exists a maximal chain $\mathcal C$ of $\spine(L)$ that contains $x$ and $y$. Since $\spine(L)$ is distributive (hence, graded), $\mathcal C$ is a maximum-length chain of $L$. It follows that $\mathcal C\cap[\hat 0,x]$ is a maximum-length chain of $[\hat 0,x]$, so $y\in\spine([\hat 0,x])$. 

We have shown that $\spine(L)\cap[\hat 0,x]\subseteq\spine([\hat 0,x])$. To prove the reverse containment, suppose $z\in\spine([\hat 0,x])$. Let $\mathcal C'$ be a maximum-length chain of $[\hat 0,x]$ that contains $z$, and let $\mathcal C''$ be a maximum-length chain of $L$ that contains $x$. Then $\mathcal C'\cup(\mathcal C''\setminus[\hat 0,x])$ is a maximum-length chain of $L$ that contains $z$, so $z\in\spine(L)$. Since $z$ is certainly in $[\hat 0,x]$, this proves that $\spine([\hat 0,x])\subseteq\spine(L)\cap[\hat 0,x]$.
\end{proof}

\begin{proof}[Proof of \Cref{thm:spine}]
Let $L$ be a trim lattice. We will prove the inequality $\mathcal E(L)\leq\mathcal E(\spine(L))$ by induction on $|L|$. This is trivial if $|L|=1$, so we may assume $|L|\geq 2$. 

To ease notation, let $K=\spine(L)$ and $Q=\cov_L(\hat 1)$. It is immediate from the definition of the spine that $\cov_K(\hat 1)=Q\cap K$. For $x\in L$ and $x'\in K$, we consider \[\Delta_L(x)=\{y\in L:y\leq x\}\quad\text{and}\quad\Delta_K(x')=\{y\in K:y\leq x'\},\] which are sublattices of $L$ and $K$, respectively. It is immediate from the definition of the Ungarian Markov chain $\bU_L$ that 
\begin{equation}\label{eq:A_1A_2}
\mathcal E(L)\left(1-(1-p)^{|Q|}\right)-1=\sum_{\emptyset\neq T\subseteq Q}p^{|T|}(1-p)^{|Q|-|T|}\,\mathcal E(\Delta_L({\textstyle\bigwedge} T))=\mathcal A_1+\mathcal A_2,
\end{equation} 
where
\[\mathcal A_1=\sum_{T'\subseteq Q\setminus K}p^{|T'|}(1-p)^{|Q\setminus K|-|T'|}\sum_{\emptyset\neq T''\subseteq Q\cap K}p^{|T''|}(1-p)^{|Q\cap K|-|T''|}\,\mathcal E(\Delta_L({\textstyle\bigwedge}(T'\cup T'')))\] and \[\mathcal A_2=\sum_{\emptyset\neq T'\subseteq Q\setminus K}p^{|T'|}(1-p)^{|Q\setminus K|-|T'|}(1-p)^{|Q\cap K|}\,\mathcal E(\Delta_L({\textstyle\bigwedge} T')).
\]
We will prove that 
\begin{equation}\label{eq:spine_eq}
\mathcal E(\Delta_L({\textstyle\bigwedge}(T'\cup T'')))\leq \mathcal E(\Delta_K({\textstyle\bigwedge} T''))
\end{equation} for all $T'\subseteq Q\setminus K$ and $T''\subseteq Q\cap K$ such that $T'\cup T''\neq\emptyset$. Because $\Delta_K(\bigwedge\emptyset)=\Delta_K(\hat 1)=K$, this will imply that
\begin{align*}
\mathcal A_1&\leq \sum_{T'\subseteq Q\setminus K}p^{|T'|}(1-p)^{|Q\setminus K|-|T'|}\sum_{\emptyset\neq T''\subseteq Q\cap K}p^{|T''|}(1-p)^{|Q\cap K|-|T''|}\,\mathcal E(\Delta_K({\textstyle\bigwedge} T'')) \\ 
&=\sum_{\emptyset\neq T''\subseteq Q\cap K}p^{|T''|}(1-p)^{|Q\cap K|-|T''|}\,\mathcal E(\Delta_K({\textstyle\bigwedge} T'')) \\ 
&=\sum_{T''\subseteq Q\cap K}p^{|T''|}(1-p)^{|Q\cap K|-|T''|}\,\mathcal E(\Delta_K({\textstyle\bigwedge} T''))-(1-p)^{|Q\cap K|}\,\mathcal E(K) \\ 
&=\mathcal E(K)-1-(1-p)^{|Q\cap K|}\mathcal E(K)
\end{align*}
and 
\begin{align*}
\mathcal A_2&\leq \sum_{\emptyset\neq  T'\subseteq Q\setminus K}p^{|T'|}(1-p)^{|Q\setminus K|-|T'|}(1-p)^{|Q\cap K|}\mathcal E(K) \\ 
&=\sum_{T'\subseteq Q\setminus K}p^{|T'|}(1-p)^{|Q\setminus K|-|T'|}(1-p)^{|Q\cap K|}\mathcal E\left(K\right)-(1-p)^{|Q\setminus K|}(1-p)^{|Q\cap K|}\mathcal E(K) \\ 
&=(1-p)^{|Q\cap K|}\mathcal E(K)-(1-p)^{|Q|}\mathcal E(K),
\end{align*}
so it will follow from \eqref{eq:A_1A_2} that \[\mathcal E(L)\left(1-(1-p)^{|Q|}\right)-1\leq\mathcal E(K)\left(1-(1-p)^{|Q|}\right)-1,\] which is equivalent to our desired inequality. 

Fix $T'\subseteq Q\setminus K$ and $T''\subseteq Q\cap K$ with $T'\cup T''\neq\emptyset$. Let $L'=\Delta_L\left(\bigwedge (T'\cup T'')\right)$. Because $L'$ is an interval in the trim lattice $L$, \Cref{prop:trim_interval} tells us that $L'$ is also trim. Let $K'=\spine(L')$. The assumption that $T'\cup T''$ is nonempty guarantees that $|L'|<|L|$, so we know by induction that $\mathcal E(L')\leq\mathcal E(K')$. Hence, in order to prove \eqref{eq:spine_eq}, it suffices to show that $\mathcal E(K')\leq\mathcal E(\Delta_K(\bigwedge T''))$. 

Because $T''\subseteq K$ and $K$ is a sublattice of $L$, we have $\bigwedge T''\in K$. According to \Cref{lem:subspine}, $\Delta_K(\bigwedge T'')$ is the spine of $\Delta_L(\bigwedge T'')$. Therefore, we know by \eqref{eq:spine_isomorphism} that \[K'\cong J(\Pos(L'))\quad\text{and}\quad \Delta_K({\textstyle\bigwedge}T'')\cong J(\Pos(\Delta_L({\textstyle\bigwedge}T''))).\] Now, $L'$ is an interval of $\Delta_L(\bigwedge T'')$, so \Cref{prop:trim_interval} tells us that $\Gal(L')$ is isomorphic to an induced subgraph of $\Gal(\Delta_L(\bigwedge T''))$. This implies that $\Pos(L')$ is a subposet of $\Pos(\Delta_L(\bigwedge T''))$, so the desired inequality $\mathcal E(K')\leq\mathcal E(\Delta_K(\bigwedge T''))$ follows from \Cref{cor:subposet}. 
\end{proof}

\begin{proof}[Proof of \Cref{cor:spine}]
Let $L$ be a trim lattice, and suppose $L'$ is an interval or a quotient of $L$. Then $L'$ is a trim lattice whose Galois graph $\Gal(L')$ is isomorphic to an induced subgraph of $\Gal(L)$; this follows from \Cref{prop:trim_interval} if $L'$ is an interval of $L$, and it follows from \cite[Lemma~3.10~\&~Remark~3.11]{ThomasWilliams} if $L'$ is a quotient of $L$. This implies that $\Pos(L')$ is a subposet of $\Pos(L)$, so we can use \eqref{eq:spine_isomorphism} and \Cref{cor:subposet} to see that $\mathcal E(\spine(L'))\leq\mathcal E(\spine(L))$. On the other hand, we can apply \Cref{thm:spine} to $L'$ to obtain the inequality $\mathcal E(L')\leq\mathcal E(\spine(L'))$. 
\end{proof}

\section{Cambrian Lattices}\label{sec:Cambrian}
In this section, we review relevant notions related to Reading's Cambrian lattices, and we analyze the spines of Cambrian lattices in order to prove \Cref{thm:CambA,thm:CambB,thm:CambD}. We also give a simple proof of \Cref{thm:dihedral}. 

\subsection{Background}\label{subsec:BackgroundCambrian}  

Let $W$ be a finite \dfn{Coxeter group}, and let $S$ be the set of \dfn{simple reflections} of $W$; this means that $W$ has a presentation of the form $\langle S:(ss')^{m(s,s')}=e\rangle$, where $e$ is the identity element of $W$, $m(s,s)=1$ for all $s\in S$, and $m(s,s')=m(s',s)\in\{2,3,\ldots\}\cup\{\infty\}$ for all distinct $s,s'\in S$.
Note that each simple reflection is an involution. The \dfn{Coxeter graph} of $W$ is the graph with vertex set $S$ in which two simple reflections $s$ and $s'$ are connected by an edge whenever $m(s,s')\geq 3$; this edge is labeled with the number $m(s,s')$ if $m(s,s')\geq 4$.
We will assume that $W$ is \dfn{irreducible}, which means that its Coxeter graph is connected (equivalently, $W$ cannot be expressed as a direct product of smaller Coxeter groups). 

When we refer to a \emph{word over $S$}, we mean a (possibly infinite) word whose letters are in $S$, where we view any two of the letters of the word as distinct from one another (even if they are copies of the same simple reflection). We can apply a \dfn{commutation move} to a word over $S$ by swapping two consecutive letters $s$ and $s'$ if $m(s,s')=2$ (we do not allow such a commutation move if $m(s,s')=1$). The \dfn{commutation class} of a word $\QQ$ over $S$ is the set of words that can be obtained from $\QQ$ via a sequence of commutation moves.
Finite words in the same commutation class represent the same element of $W$. Following Viennot \cite{Viennot}, we define a certain poset $\Heap(\QQ)$ called the \dfn{heap} of $\QQ$. The elements of this poset are the letters in $\QQ$ (which are seen as distinct from one another). The order relation is defined so that if $s$ and $s'$ are two letters (which could represent the same simple reflection), then $s<s'$ if and only if $s$ appears to the left of $s'$ in every word in the commutation class of $\QQ$. The Hasse diagram of $\Heap(\QQ)$, which we will also denote by $\Heap(\QQ)$, is called the \dfn{combinatorial AR quiver} of $\QQ$ \cite[Chapter~9]{WhyTheFuss} and is typically drawn sideways so that each cover relation $s\lessdot s'$ is depicted with $s$ to the left of $s'$; see subsequent subsections for several such situations. 

A \dfn{reduced word} for an element $w\in W$ is a word over $S$ that represents $w$ and has minimum length among all such words. The (right) \dfn{weak order} on $W$ is the partial order on $W$ defined by declaring $u\leq v$ if there is a reduced word for $v$ that contains a reduced word for $u$ as a prefix. We have assumed that $W$ is finite, so a seminal result due to Bj\"orner \cite{Bjorner} states that the weak order on $W$ is a lattice; the maximal element of this lattice is called the \dfn{long element} of $W$ and is typically denoted $w_0$. 

Let $s_{i_1},\ldots,s_{i_n}$ be an ordering of the simple reflections of $W$. Let $\cc$ be the word $s_{i_1}\cdots s_{i_n}$. The element $c$ of $W$ represented by $\cc$ is called a \dfn{Coxeter element} of $W$. Two words $\cc$ and $\cc'$ represent the same Coxeter element if and only if $\Heap(\cc)=\Heap(\cc')$. The Coxeter graph of $W$ is a tree. Let us orient each edge $\{s,s'\}$ in the Coxeter graph from $s$ to $s'$ if $s$ appears to the left of $s'$ in $\cc$ (equivalently, $s\leq s'$ in $\Heap(\cc)$), and let us orient it from $s'$ to $s$ if $s$ appears to the right of $s'$ in $\cc$ (equivalently, $s\geq s'$ in $\Heap(\cc)$). This orientation depends only on $c$. In fact, this establishes a one-to-one correspondence between Coxeter elements of $W$ and orientations of the Coxeter graph of $W$. 

Let $\cc^\infty=\cc^{\langle 1\rangle}\cc^{\langle 2\rangle}\cdots$, where each word $\cc^{\langle i\rangle}$ is a copy of $\cc$. Following Reading \cite{ReadingSortable}, we define the \dfn{$\cc$-sorting word} of an element $w\in W$, denoted $\word_{\cc}(w)$, to be the reduced word for $w$ that is lexicographically first as a subword of $\cc^\infty$. We can write $\word_{\cc}(w)=\mathsf{w}^{\langle 1\rangle}\mathsf{w}^{\langle 2\rangle}\cdots$, where $\mathsf{w}^{\langle i\rangle}$ is the subword of $\word_{\cc}(w)$ that came from $\cc^{\langle i\rangle}$ when we found $\word_{\cc}(w)$ as a lexicographically minimum subword of $\cc^\infty$. Let $\text{supp}(\mathsf{w}^{\langle i\rangle})$ be the set of simple reflections appearing in $\mathsf{w}^{\langle i\rangle}$. We say $w$ is \dfn{$c$-sortable} if we have the chain of containments $\text{supp}(\mathsf{w}^{\langle 1\rangle})\supseteq \text{supp}(\mathsf{w}^{\langle 2\rangle})\supseteq \text{supp}(\mathsf{w}^{\langle 3\rangle})\supseteq\cdots$. Whether or not $w$ is $c$-sortable depends only on the Coxeter element $c$ and not on the word $\cc$ \cite{ReadingSortable}. 

The set of $c$-sortable elements of $W$ forms a sublattice (and a quotient lattice) $\Camb_c$ of the weak order on $W$ called the \dfn{$c$-Cambrian lattice}. Moreover, $\Camb_c$ is trim \cite{Thomas, Ingalls}.
Thomas and Williams \cite{ThomasWilliams} provided a description of the Galois graph $\Gal(\Camb_c)$, from which one can deduce that the Galois poset of the $c$-Cambrian lattice is isomorphic to the heap of the $c$-sorting word of the long element of $W$: 
\begin{equation}\label{eq:Pos(Camb)}
\Pos(\Camb_c)\cong\Heap(\word_{\cc}(w_0)).
\end{equation}

All Coxeter elements of $W$ are conjugate to each other, so they all have the same group-theoretic order $h$, which is called the \dfn{Coxeter number} of $W$.
Let $\cc^h$ be the word obtained by concatenating $\cc$ with itself $h$ times. The long element $w_0$ is an involution in $W$, so it gives rise to an involution $\psi\colon S\to S$ defined by $\psi(s)=w_0sw_0$. We can extend $\psi$ to a map defined on words over $S$ by letting $\psi(s_{i_1}\cdots s_{i_k})=\psi(s_{i_1})\cdots\psi(s_{i_k})$ (and similarly for infinite words). It follows from \cite[Lemma~2.6.5]{WhyTheFuss} that
\begin{equation}\label{eq:psi}
\word_{\cc}(w_0)\psi(\word_{\cc}(w_0))=\cc^h.
\end{equation}

 \subsection{Cambrian Lattices of Type $A$}\label{subsec:CambrianA}
The purpose of this subsection is to prove \Cref{thm:CambA}, which provides an asymptotic upper bound for $\mathcal E(L)$ when $L$ is a (large) Cambrian lattice of type~$A$.

The Coxeter group $A_n$ is the symmetric group $S_{n+1}$; its simple reflections are $s_1,\ldots,s_n$, where $s_i$ is the transposition $(i\,\, i+1)$. The Coxeter graph of $A_n$ is a path with edges $\{s_i,s_{i+1}\}$ for $i\in[n-1]$, and each of these edges is unlabeled (i.e., $m(s_i,s_{i+1})=3$). The Coxeter number of $A_n$ is $n+1$. 

Fix a reduced word $\cc$ of a Coxeter element $c$ of $A_n$. The isomorphisms \eqref{eq:spine_isomorphism} and \eqref{eq:Pos(Camb)} tell us that the spine of $\Camb_c$ is isomorphic to $J(\Heap(\word_{\cc}(w_0)))$, so we wish to describe the combinatorial AR quiver of $\word_{\cc}(w_0)$. We refer the reader to \Cref{exam:CambA} and \Cref{fig:CambA1,fig:CambA2} for concrete illustrations the following discussion. 

Begin by drawing $\Heap(\cc^{n+1})$ by drawing $n+1$ copies of $\Heap(\cc)$ in a row and adding in edges as appropriate so that the result has the shape of a chain-link fence. We can coordinatize this drawing in the $xy$-plane so that the leftmost point has $x$-coordinate $0$, each edge extends $1$ unit horizontally and $1$ unit vertically, and all of the vertices that are copies of $s_i$ lie on the line $y=i$. Draw a path that has the same shape as an upside-down version of $\Heap(\cc)$, and use this path to cut some of the edges in $\Heap(\cc^{n+1})$ so that exactly half of the $n(n+1)$ vertices are to the left of the cut (the remainder of this paragraph will imply that this is indeed possible). The involution $\psi\colon S\to S$ is given by $\psi(s_i)=s_{n+1-i}$. This implies that $\Heap(\psi(\word_{\cc}(w_0)))$ has the same ``shape'' as $\Heap(\word_{\cc}(w_0))$, except that it is flipped upside-down, and \eqref{eq:psi} tells us that if we place $\Heap(\word_{\cc}(w_0))$ immediately to the left of $\Heap(\psi(\word_{\cc}(w_0)))$ and add edges as appropriate, we will obtain $\Heap(\cc^{n+1})$. Hence, the path that we drew to make the cut has the same shape as $\Heap(\psi(\cc))$, and it follows from \eqref{eq:psi} that the combinatorial AR quiver $\Heap(\word_{\cc}(w_0))$ is the part of the graph on the left of the cut. 

In our drawing of $\Heap(\word_{\cc}(w_0))$, if $(a,1)$ and $(b,n)$ are the leftmost points in the bottom and top rows, respectively, then $(n-1+b,1)$ and $(n-1+a,n)$ are the rightmost points in the bottom and top rows, respectively. It follows that this drawing of $\Heap(\word_{\cc}(w_0))$ lies within the square $\mathscr R_{\cc}$ whose sides lie on the lines $y=x+n-b$, $y=x-n+2-b$, $y=-x+a+1$, and $y=-x+2n-1+a$. 

\begin{example}\label{exam:CambA}
Suppose $n=9$. Let $\cc$ be the word $s_3s_2s_1s_4s_5s_7s_6s_8s_9$, and let $c\in A_9$ be the element represented by $\cc$. \Cref{fig:CambA1} shows how to arrange $n+1=10$ copies of $\Heap(\cc)$ (drawn in different colors) and add edges (drawn in black) to obtain $\Heap(\cc^{10})$. The thick black-and-blue path has the same shape as $\Heap(\psi(\cc))$ and cuts $\Heap(\cc^{10})$ into two halves; the left half is $\Heap(\word_{\cc}(w_0))$. Indeed, the $\cc$-sorting word for $w_0$ is \[{\color{red}s_3s_2s_1s_4s_5s_7s_6s_8s_9} 
{\color{NormalGreen}s_3s_2s_1s_4s_5s_7s_6s_8s_9}{\color{ChillBlue}s_3s_2s_1s_4s_5s_7s_6s_8s_9}{\color{MyOrange}s_3s_2s_1s_4s_5s_7s_6s_8s_9}{\color{LimeGreen}s_3s_2s_1s_4s_5} {\color{MyPurple}s_3s_2s_1s_4}.\]
In the notation from above, we have $a=2$ and $b=4$. \Cref{fig:CambA2} shows that the square $\mathscr R_{\cc}$ (drawn dotted) whose sides lie on the lines $y=x+5$, $y=x-11$, $y=-x+3$, and $y=-x+19$ contains our drawing of $\Heap(\word_{\cc}(w_0))$. 
\end{example}

\begin{figure}[ht]
  \begin{center}{\includegraphics[height=6.9cm]{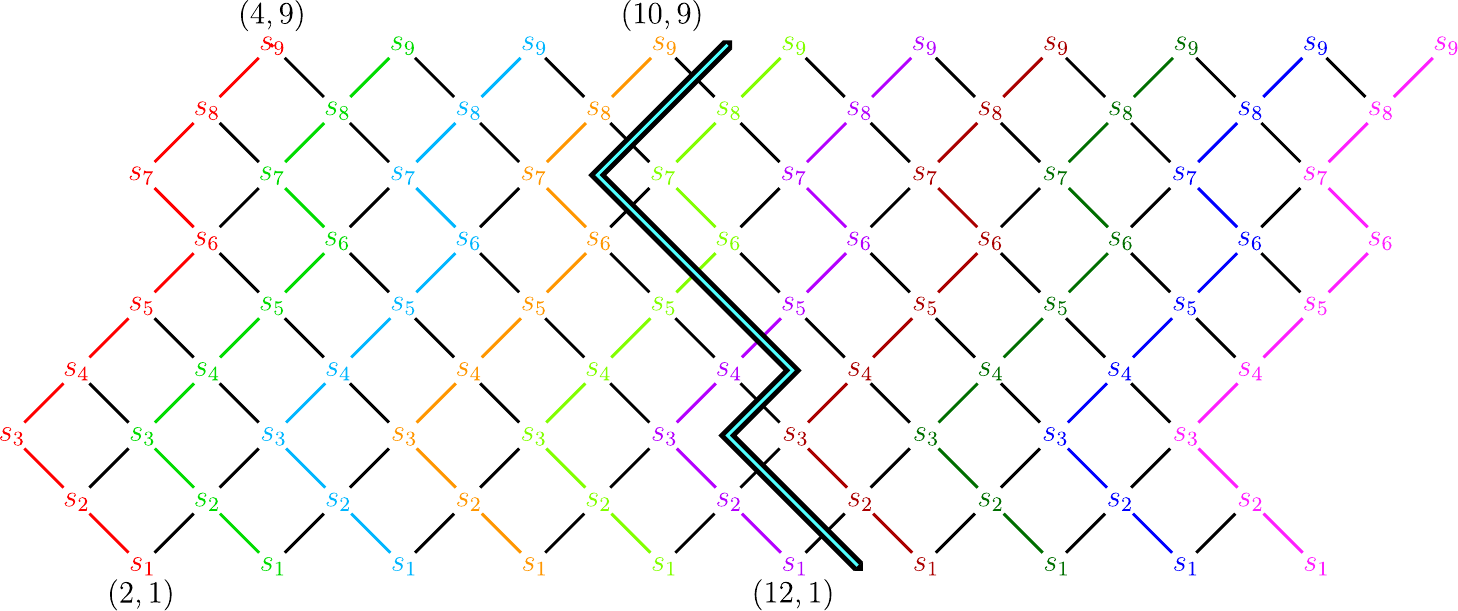}}
  \end{center}
  \caption{A drawing of the (type~$A$) combinatorial AR quiver $\Heap(\cc^{10})$, where $\cc=s_3s_2s_1s_4s_5s_7s_6s_8s_9$ is the word from \Cref{exam:CambA}. The thick black-and-blue path cuts $\Heap(\cc^{10})$ into a left half, which is $\Heap(\word_{\cc}(w_0))$, and a right half, which is $\Heap(\psi(\word_{\cc}(w_0)))$.}\label{fig:CambA1}
\end{figure}

\begin{figure}[ht]
  \begin{center}{\includegraphics[height=11.805cm]{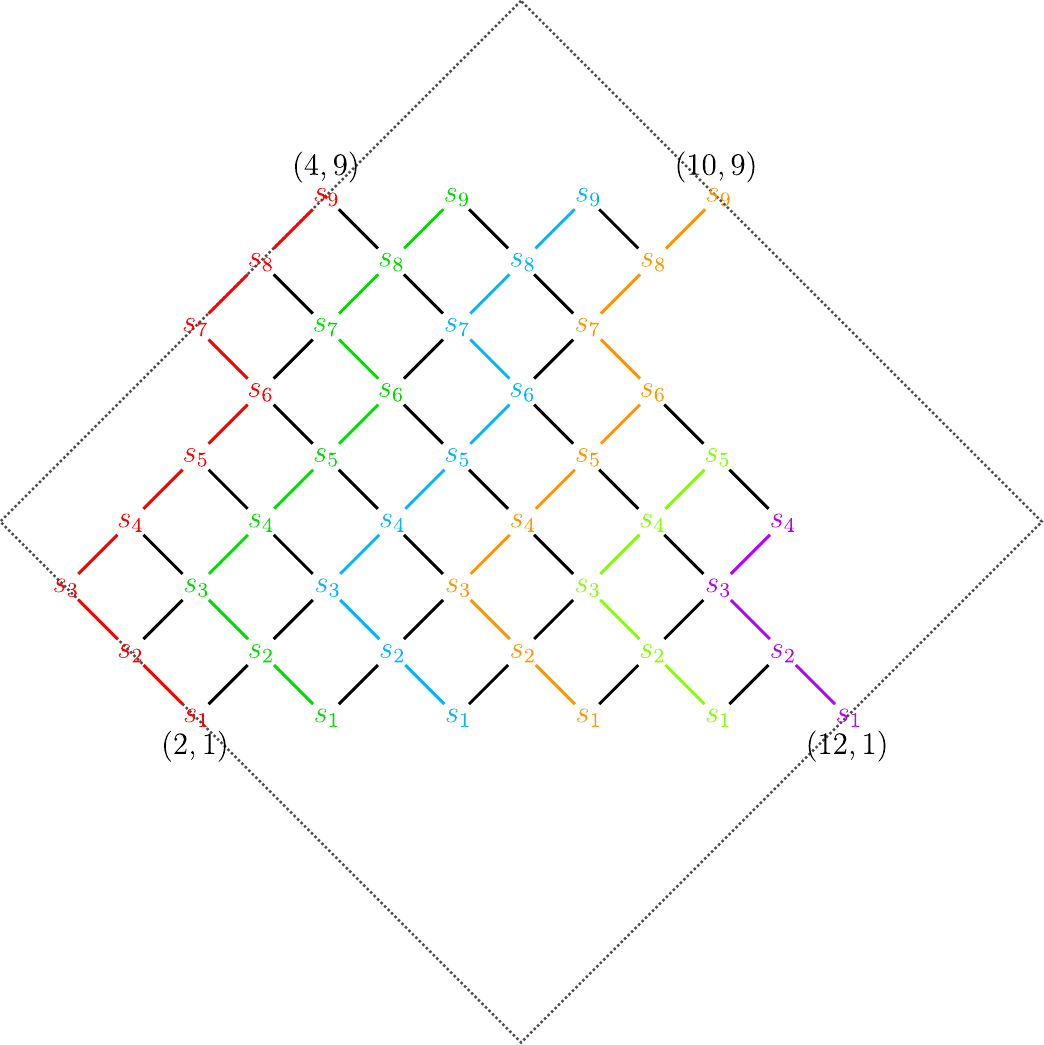}}
  \end{center}
  \caption{The (type~$A$) combinatorial AR quiver $\Heap(\word_{\cc}(w_0))$ from \Cref{exam:CambA} fits inside the square $\mathscr R_{\cc}$.}\label{fig:CambA2}
\end{figure}

The fact that our drawing of $\Heap(\word_{\cc}(w_0))$ fits inside the square $\mathscr R_{\cc}$ readily implies that $\Heap(\word_{\cc}(w_0))$ is a subposet of the $n\times n$ rectangle poset $R_{n\times n}$. It follows from \Cref{thm:rectangle,cor:subposet} that \[\mathcal E(J(\Heap(\word_{\cc}(w_0))))\leq\mathcal E(J(R_{n\times n}))=\frac{1}{p}\left(2+2\sqrt{1-p}\right)n+o(n).\] On the other hand, \eqref{eq:spine_isomorphism} and \eqref{eq:Pos(Camb)} tell us that $J(\Heap(\word_{\cc}(w_0)))$ is isomorphic to the spine of $\Camb_c$. Invoking \Cref{thm:spine}, we find that \[\mathcal E(\Camb_c)\leq \frac{1}{p}\left(2+2\sqrt{1-p}\right)n+o(n),\] and this proves \Cref{thm:CambA}. 

\subsection{Cambrian Lattices of Type $B$}\label{subsec:CambrianB}
This subsection is devoted to proving \Cref{thm:CambB}, which provides an asymptotic upper bound for $\mathcal E(L)$ when $L$ is a (large) Cambrian lattice of type~$B$. 

The Coxeter group $B_n$ is the $n$-th \dfn{hyperoctahedral group}; it has simple reflections $s_0,s_1,\ldots,s_{n-1}$. The Coxeter graph of $B_n$ is a path with edges $\{s_i,s_{i+1}\}$ for $i\in\{0,\ldots,n-2\}$. The label of the edge $\{s_0,s_1\}$ is $m(s_0,s_1)=4$, while each edge $\{s_i,s_{i+1}\}$ with $1\leq i\leq n-2$ is unlabeled (i.e., $m(s_i,s_{i+1})=3$). The Coxeter number of $B_n$ is $2n$. 

Fix a reduced word $\cc$ of a Coxeter element $c$ of $B_n$. As discussed in \Cref{subsec:BackgroundCambrian}, there is a unique orientation of the Coxeter graph of $B_n$ corresponding to $c$; let $r(c)$ be the number of edges $\{s_i,s_{i+1}\}$ that are oriented from $s_i$ to $s_{i+1}$ in this orientation. 

Finding the $\cc$-sorting word for $w_0$ is actually quite simple because the involution $\psi\colon S\to S$ is the identity map. Therefore, it follows from \eqref{eq:psi} that $\word_{\cc}(w_0)=\cc^{h/2}=\cc^n$. To draw $\Heap(\cc^n)$, simply draw $n$ copies of $\Heap(\cc)$ in a row and add in edges as appropriate so that the result has the shape of a chain-link fence. Let us coordinatize this drawing in the $xy$-plane so that the leftmost point has $x$-coordinate $0$, each edge extends $1$ unit horizontally and $1$ unit vertically, and all of the vertices that are copies of $s_i$ lie on the line $y=i$. It is a straightforward exercise to show that there exist integers $q_1$ and $q_2$ such that this drawing of $\Heap(\cc^n)$ lies in the rectangle $\mathscr R_{\cc}$ whose sides lie on the lines $y=x+q_1$, $y=x+q_1-4n+4+2r(c)$, $y=-x+q_2$, and $y=-x+q_2-2n+2-2r(c)$.   

\begin{example}\label{exam:CambB}
Suppose $n=7$. Let $\cc=s_1s_0s_2s_3s_5s_4s_6$, and let $c\in B_7$ be the element represented by $\cc$. The orientation of the Coxeter graph of $B_7$ corresponding to $c$ is
\[\begin{array}{l}\includegraphics[height=0.43cm]{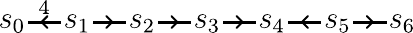}\end{array}.\] 
\Cref{fig:CambB} shows how to arrange $7$ copies of $\Heap(\cc)$ (drawn in different colors) and add edges (drawn in black) to obtain $\Heap(\cc^{7})=\Heap(\word_{\cc}(w_0))$. In the notation from above, we have $r(c)=4$, $q_1=4$ and $q_2=22$. \Cref{fig:CambB} shows that the square $\mathscr R_{\cc}$ (drawn dotted) whose sides lie on the lines $y=x+4$, $y=x-12$, $y=-x+22$, and $y=-x+2$ contains our drawing of $\Heap(\word_{\cc}(w_0))$.
\end{example} 

\begin{figure}[ht]
  \begin{center}{\includegraphics[height=13.258cm]{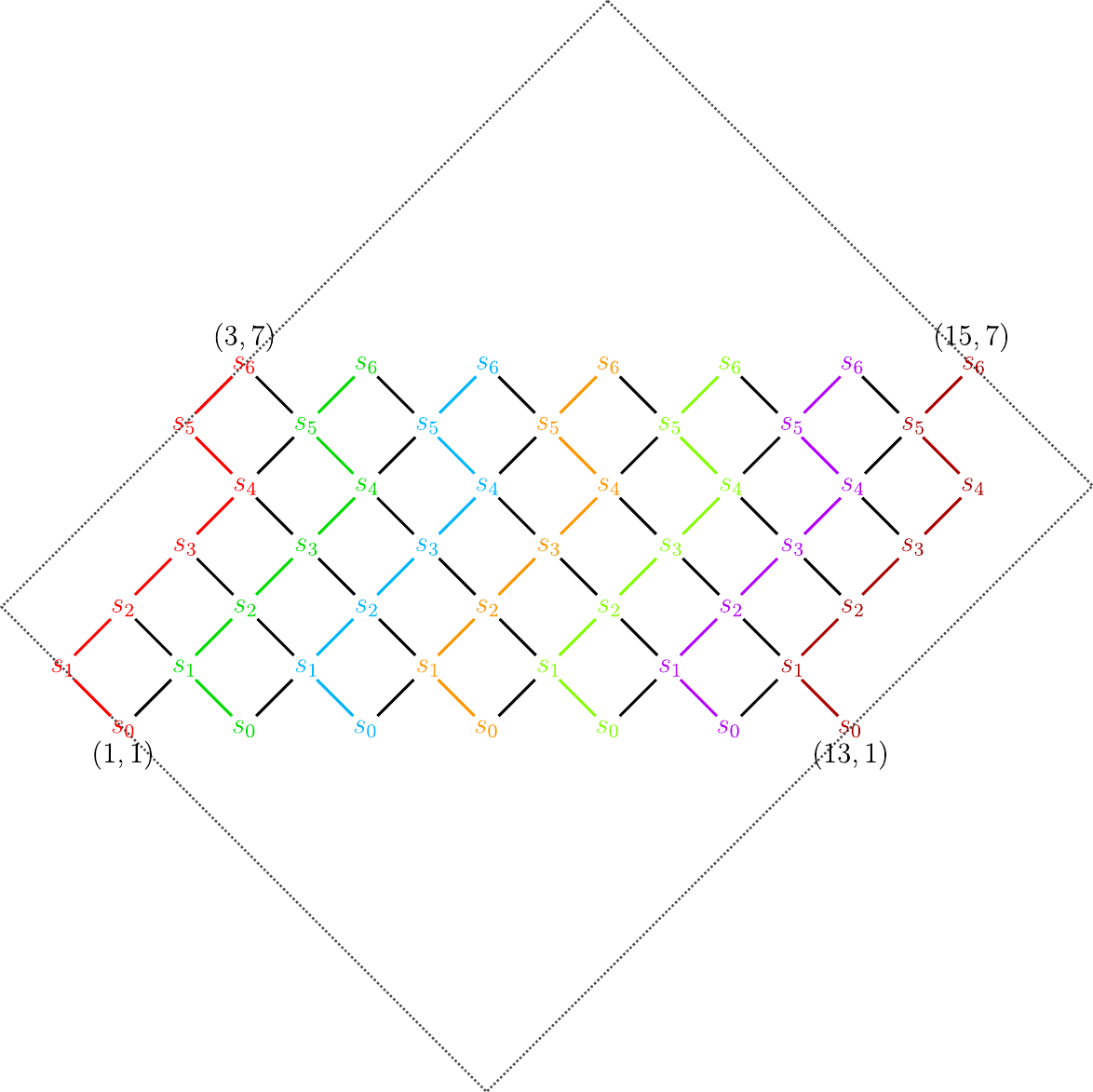}}
  \end{center}
  \caption{The (type~$B$) combinatorial AR quiver $\Heap(\word_{\cc}(w_0))=\Heap(\cc^7)$, where $\cc=s_1s_0s_2s_3s_5s_4s_6$ is the word from \Cref{exam:CambB}. This drawing fits inside the rectangle $\mathscr R_{\cc}$.}\label{fig:CambB}
\end{figure}

As in the statement of \Cref{thm:CambB}, let us now choose a sequence $(c^{(n)})_{n\geq 2}$, where each $c^{(n)}$ is a Coxeter element of $B_n$, and let us assume that the limit $\overline{r}=\lim\limits_{n\to\infty}\frac{1}{n}r(c^{(n)})$ exists. For each $n$, let $\cc^{(n)}$ be a reduced word for $c^{(n)}$. Because our drawing of $\Heap(\word_{\cc^{(n)}}(w_0))$ fits inside the rectangle $\mathscr R_{\cc^{(n)}}$, $\Heap(\word_{\cc^{(n)}}(w_0))$ is a subposet of the $(2n-1-r(c^{(n)}))\times (n+r(c^{(n)}))$ rectangle poset $R_{(2n-1-r(c^{(n)}))\times (n+r(c^{(n)}))}$. It follows from \Cref{thm:rectangle,cor:subposet} that 
\begin{align*}
\mathcal E(J(\Heap(\word_{\cc^{(n)}}(w_0))))&\leq\mathcal E\!\left(J\!\left(R_{(2n-1-r(c^{(n)}))\times (n+r(c^{(n)}))}\right)\right) \\ 
&=\frac{1}{p}\left(3+2\sqrt{(1-p)(2-\overline r)(1+\overline{r})}\right)n+o(n).
\end{align*} 
On the other hand, \eqref{eq:spine_isomorphism} and \eqref{eq:Pos(Camb)} tell us that $J(\Heap(\word_{\cc^{(n)}}(w_0)))$ is isomorphic to the spine of $\Camb_{c^{(n)}}$. Invoking \Cref{thm:spine}, we find that \[\mathcal E(\Camb_{c^{(n)}})\leq \frac{1}{p}\left(3+2\sqrt{(1-p)(2-\overline r)(1+\overline{r})}\right)n+o(n),\] and this proves \Cref{thm:CambB}. 

\subsection{Cambrian Lattices of Type $D$}
In this subsection, we prove \Cref{thm:CambD}, which provides an asymptotic upper bound for $\mathcal E(L)$ when $L$ is a (large) Cambrian lattice of type~$D$. 

The Coxeter group $D_n$ has simple reflections $s_0,s_1,\ldots,s_{n-1}$. The Coxeter graph of $D_n$ has unlabeled edges $\{s_i,s_{i+1}\}$ for $i\in[n-2]$ together with the additional unlabeled edge $\{s_0,s_2\}$. The Coxeter number of $D_n$ is $2n-2$. 

Consider a reduced word $\cc$ of a Coxeter element $c$ of $D_n$. As discussed in \Cref{subsec:BackgroundCambrian}, there is a unique orientation of the Coxeter graph of $D_n$ corresponding to $c$; let $r(c)$ be the number of edges of the form $\{s_i,s_{i+1}\}$ with $1\leq i\leq n-2$ that are oriented from $s_i$ to $s_{i+1}$ in this orientation (so the definition of $r(c)$ ignores the edge $\{s_0,s_2\}$). 

If $n$ is even, the involution $\psi\colon S\to S$ is the identity map, but if $n$ is odd, this involution is the transposition that fixes $s_2,s_3,\ldots,s_{n-1}$ and swaps $s_0$ and $s_1$. It is certainly possible to prove \Cref{thm:CambD} by describing $\word_{\cc}(w_0)$ in a manner that depends on the parity of $n$. However, because \Cref{thm:CambD} is only concerned with asymptotics, we can take a shortcut by reducing our analysis to the case when $n$ is even. We do so via the following lemma, which makes tacit use of the standard identification of $D_n$ with the subgroup of $D_{n+1}$ generated by $s_0,s_1,\ldots,s_{n-1}$.  

\begin{lemma}\label{lem:shortcut}
Suppose $n\geq 4$. Let $c$ be a Coxeter element of $D_n$, and let $c'$ be the Coxeter element $s_nc$ of $D_{n+1}$. Then $r(c)=r(c')$, and \[\mathcal E(\Camb_c)\leq\mathcal E(\spine(\Camb_{c'})).\]
\end{lemma}

\begin{proof}
The identity $r(c)=r(c')$ is immediate from the relevant definitions. Let $w_0(\{s_0,\ldots,s_{n-1}\})$ be the long element of $D_n$, seen as an element of $D_{n+1}$; it is known that this element is $c'$-sortable. It is straightforward to show that an element $v\in D_{n+1}$ is a $c$-sortable element of $D_n$ if and only if $v$ is $c'$-sortable and $v\leq w_0(\{s_0,\ldots,s_{n-1}\})$ in the weak order on $D_{n+1}$. Consequently, $\Camb_c$ is the interval $[e,w_0(\{s_0,\ldots,s_{n-1}\})]$ of $\Camb_{c'}$. The desired inequality now follows from \Cref{cor:spine}.    
\end{proof}

Let us now assume $n$ is even so that $\psi$ is the identity map. Fix a reduced word $\cc$ of a Coxeter element $c\in D_n$. It follows from \eqref{eq:psi} that $\word_{\cc}(w_0)=\cc^{h/2}=\cc^{n-1}$. We will draw $\Heap(\cc^{n-1})$ in the $xy$-plane as follows (see \Cref{exam:CambD} and \Cref{fig:CambD} for an illustration). Draw $n-1$ copies of $\Heap(\cc)$ in a row and add in edges as appropriate. For each $i\in[n-1]$, draw all of the vertices that are copies of $s_i$ on the line $y=i$. Also, draw all of the vertices that are copies of $s_0$ on the line $y=2$. Assume the leftmost point in the drawing has $x$-coordinate $0$. Draw each edge that does not include $s_0$ as a line segment that extends $1$ unit horizontally and $1$ unit vertically, and draw each edge that does include $s_0$ as a horizontal line segment of length $1$. 

As in \Cref{subsec:CambrianA,subsec:CambrianB}, we wish to find a rectangle $\mathscr R_{\cc}$ that encloses our drawing of the combinatorial AR quiver $\Heap(\word_{\cc}(w_0))=\Heap(\cc^{n-1})$. Describing this rectangle is a bit cumbersome because of the presence of vertices that are copies of $s_0$. However, these vertices can only increase the length and width of the rectangle by a small amount that will be absorbed by the error term in our asymptotic results anyway. To be more precise, consider the orientation of the Coxeter graph of $D_n$ corresponding to $c$. Let $\epsilon_1=1$ if one of the following conditions holds: 
\begin{itemize}
\item $\{s_0,s_2\}$ is oriented from $s_2$ to $s_0$, and $\{s_1,s_2\}$ is oriented from $s_1$ to $s_2$; 
\item $\{s_0,s_2\}$ is oriented from $s_0$ to $s_2$, and $\{s_i,s_{i+1}\}$ is oriented from $s_i$ to $s_{i+1}$ for all $2\leq i\leq n-2$;
\end{itemize}
otherwise, let $\epsilon_1=0$. 
Let $\epsilon_2=1$ if one of the following conditions holds: 
\begin{itemize}
\item $\{s_0,s_2\}$ is oriented from $s_0$ to $s_2$, and $\{s_1,s_2\}$ is oriented from $s_2$ to $s_1$; 
\item $\{s_0,s_2\}$ is oriented from $s_2$ to $s_0$, and $\{s_i,s_{i+1}\}$ is oriented from $s_{i+1}$ to $s_i$ for all $2\leq i\leq n-2$;
\end{itemize}
otherwise, let $\epsilon_2=0$. Then there are integers $q_1$ and $q_2$ such that the sides of $\mathscr R_{c}$ lie on the lines $y=x+q_1$, $y=x+q_1-4n+8+2r(c)-\epsilon_1$, $y=-x+q_2$, and $y=-x+q_2-2n+4-2r(c)-\epsilon_2$. 

\begin{example}\label{exam:CambD}
Suppose $n=6$. Let $\cc=s_0s_3s_2s_1s_5s_4$, and let $c\in D_6$ be the element represented by $\cc$. The orientation of the Coxeter graph of $D_6$ corresponding to $c$ is \[\begin{array}{l}\includegraphics[height=1.011cm]{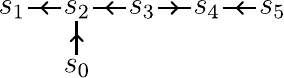}\end{array}.\] 
\Cref{fig:CambD} shows how to arrange $5$ copies of $\Heap(\cc)$ (drawn in different colors) and add edges (drawn in black) to obtain $\Heap(\cc^5)=\Heap(\word_{\cc}(w_0))$. In the notation from above, we have $\epsilon_1=0$, $\epsilon_2=1$, $r(c)=1$, $q_1=5$, and $q_2=13$. \Cref{fig:CambD} shows that the rectangle $\mathscr R_{\cc}$ (drawn dotted) whose sides lie on the lines $y=x+5$, $y=x-9$, $y=-x+13$, and $y=-x+2$ contains our drawing of $\Heap(\word_{\cc}(w_0))$. \Cref{fig:CambD} also uses additional diagonal grid lines to illustrate how $\Heap(\word_{\cc}(w_0))$ is a subposet of the rectangle poset $R_{(4n-7-2r(c)+\epsilon_1)\times(2n-3+2r(c)+\epsilon_2)}=R_{15\times 12}$. 
\end{example} 

\begin{figure}[ht]
  \begin{center}{\includegraphics[height=9.214cm]{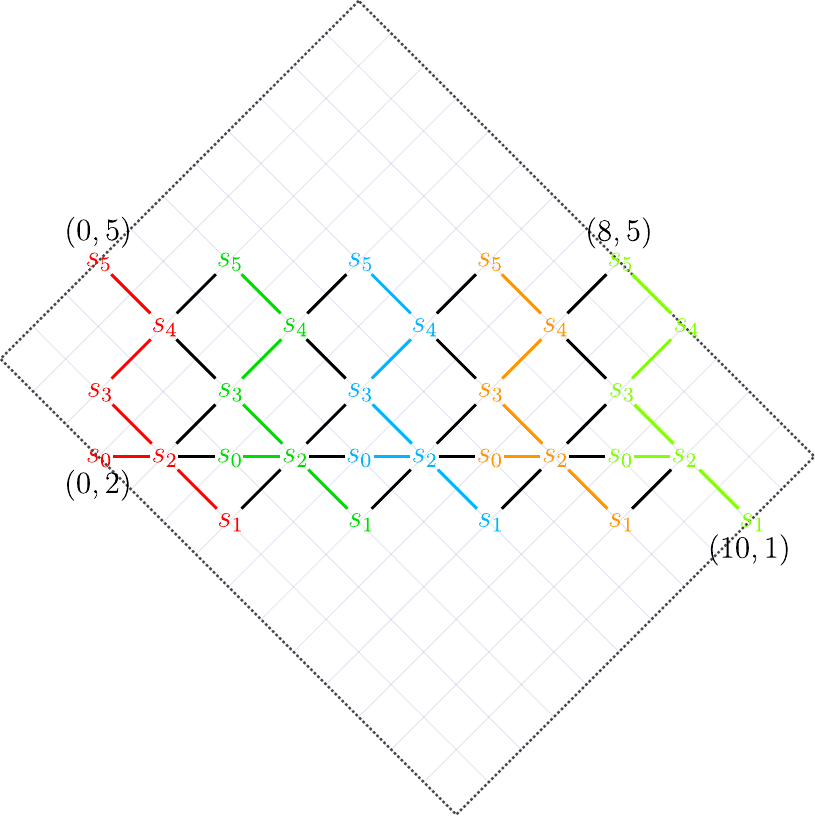}}
  \end{center}
  \caption{The (type~$D$) combinatorial AR quiver $\Heap(\word_{\cc}(w_0))=\Heap(\cc^5)$, where $\cc=s_0s_3s_2s_1s_5s_4$ is the word from \Cref{exam:CambD}. This drawing fits inside the rectangle $\mathscr R_{\cc}$. The extra diagonal grid lines serve to reveal how $\Heap(\word_{\cc}(w_0))$ is a subposet of the rectangle poset $R_{15\times 12}$.}\label{fig:CambD}
\end{figure}

As illustrated in \Cref{fig:CambD}, the fact that our drawing of $\Heap(\word_{\cc}(w_0))$ fits inside $\mathscr R_{\cc}$ allows us to see that $\Heap(\word_{\cc}(w_0))$ is a subposet of the rectangle poset $R_{(4n-7-2r(c)+\epsilon_1)\times(2n-3+2r(c)+\epsilon_2)}$. Hence, $\Heap(\word_{\cc}(w_0))$ is a subposet of $R_{(4n-2r(c))\times (2n+2r(c))}$.

As in the statement of \Cref{thm:CambD}, let us now choose a sequence $(c^{(n)})_{n\geq 4}$, where each $c^{(n)}$ is a Coxeter element of $D_n$. For each $n$, let $\cc^{(n)}$ be a reduced word for $c^{(n)}$. Let us assume that the limits $\overline{r}=\lim\limits_{n\to\infty}\frac{1}{n}r(c^{(n)})$ and $\overline{u}=\lim\limits_{n\to\infty}\frac{1}{n}u(c^{(n)})$ exist, where $u(c^{(n)})$ is the maximum number of edges in a directed path in the orientation of the Coxeter graph of $D_n$ corresponding to $c^{(n)}$ (equivalently, $u(c^{(n)})$ is the length of the poset $\Heap(\cc^{(n)})$). When $n$ is even, we have seen that $\Heap(\word_{\cc^{(n)}}(w_0))$ is a subposet of $R_{(4n-2r(c^{(n)}))\times (2n+2r(c^{(n)}))}$; it follows from \Cref{thm:rectangle,cor:subposet} that 
\begin{align*}
\mathcal E(J(\Heap(\word_{\cc^{(n)}}(w_0))))&\leq\mathcal E(J(R_{(4n-2r(c^{(n)}))\times (2n+2r(c^{(n)}))})) \\ 
&=\frac{1}{p}\left(6+4\sqrt{(1-p)(2-\overline r)(1+\overline{r})}\right)n+o(n).
\end{align*} 
We also know by \eqref{eq:spine_isomorphism} and \eqref{eq:Pos(Camb)} that $J(\Heap(\word_{\cc^{(n)}}(w_0)))$ is isomorphic to the spine of $\Camb_{c^{(n)}}$. Invoking \Cref{thm:spine}, we find that \[\mathcal E(\Camb_{c^{(n)}})\leq \frac{1}{p}\left(6+4\sqrt{(1-p)(2-\overline r)(1+\overline{r})}\right)n+o(n)\] for even $n$. Using \Cref{lem:shortcut}, we find that the same inequality holds for odd $n$.  

To complete the proof of \Cref{thm:CambD}, we must show that \[\mathcal E(\Camb_{c^{(n)}})\leq \frac{1}{p}\left(2+\overline{u}+\log\left(5\cdot 2^{\overline{u}}\right)+\sqrt{2(2+\overline{u})\log \left(5\cdot 2^{\overline{u}}\right)+\left(\log \left(5\cdot 2^{\overline{u}}\right)\right)^2}\right)n+o(n).\] Because $J(\Heap(\word_{\cc}(w_0)))\cong\spine(\Camb_{c^{(n)}})$, we can do so by appealing to \Cref{thm:distributive_Chernoff,thm:spine}. For even $n$, we need to prove that the maximum size of a chain in $\Heap(\word_{\cc^{(n)}}(w_0))$ is at most $(2+\overline{u}+o(1))n$ and that \[\left\lvert\MC(\Heap(\word_{\cc^{(n)}}(w_0)))\right\rvert\leq\left(5\cdot 2^{\overline{u}}\right)^{(1+o(1))n};\] the result for odd $n$ will then follow from \Cref{lem:shortcut}. Suppose $n$ is even. The statement about the maximum size of a chain is immediate; indeed, it follows from the above discussion that the maximum size of a chain in $\Heap(\word_{\cc^{(n)}}(w_0))$ is $2n-3+u(c^{(n)})$. Now imagine constructing a maximal chain $\mathcal C$ in $\Heap(\word_{\cc^{(n)}}(w_0))$ that uses exactly $k$ copies of $s_0$. Specifying these $k$ copies will also determine $2k-1$ or $2k$ of the edges used in $\mathcal C$; we must then choose the remaining edges. The number of remaining edges is at most $2n-3+u(c^{(n)})-2k$. If we choose these remaining edges one-by-one, starting at the left of the combinatorial AR quiver $\Heap(\word_{\cc^{(n)}}(w_0))$, then we will have at most 2 choices for each such edge. This shows that the total number of choices for $\mathcal C$ is at most $\binom{n-1}{k}2^{2n-3+u(c^{(n)})-2k}$. Consequently, 
\begin{align*}
\left\lvert\MC(\Heap(\word_{\cc^{(n)}}(w_0)))\right\rvert&\leq\sum_{k=0}^{n-1}\binom{n-1}{k}2^{2n-3+u(c^{(n)})-2k} \\ 
&=2^{(2+\overline{u})n+o(n)}\sum_{k=0}^{n-1}\binom{n-1}{k}2^{-2k} \\ 
&=2^{(2+\overline{u})n+o(n)}(5/4)^{n-1} \\ 
&=\left(5\cdot 2^{\overline u}\right)^{(1+o(1))n}.
\end{align*}

\subsection{Cambrian Lattices of Dihedral Type}
The dihedral group of order $2m$, denoted $I_2(m)$, is a Coxeter group with two simple generators $s$ and $t$ such that $m(s,t)=m$. This group has two Coxeter elements: $st$ and $ts$. The Cambrian lattices for these Coxeter elements are isomorphic, so we may assume $c=st$. Let $[s\vert t]_k$ denote the word of length $k$ that starts with $s$ and alternates between $s$ and $t$. Let $\alpha_k$ be the element of $I_2(m)$ represented by $[s\vert t]_k$. For example, $\alpha_3=sts$, and $\alpha_4=stst$. The $c$-Cambrian lattice consists of the two chains $e\lessdot \alpha_1\lessdot \alpha_2\lessdot \cdots\lessdot\alpha_m$ and $e\lessdot t\lessdot\alpha_m$. If $\mathcal C$ is a chain of length $r$, then $\mathcal E(\mathcal C)=r/p$. It follows immediately from the definition of the Ungarian Markov chain $\bU_{\Camb_{c}}$ that \[\mathcal E(\Camb_c)=p^2\cdot 1+p(1-p)\cdot(1+(m-1)/p)+p(1-p)\cdot(1+1/p)+(1-p)^2\cdot(1+\mathcal E(\Camb_c)).\] Solving this equation yields \[\mathcal E(\Camb_c)=\frac{1+m(1-p)}{2p-p^2},\] which proves \Cref{thm:dihedral}.

\section{$\nu$-Tamari Lattices}\label{sec:nu-Tamari}
In this section, we define the $\nu$-Tamari lattice $\Tam(\nu)$ of Pr\'eville-Ratelle and Viennot, and we give a simple description of its spine. This will allow us to use \Cref{thm:spine} to prove the upper bound for $\mathcal E(\Tam(\nu))$ stated in \Cref{thm:nu-Tamari}. 

A \dfn{lattice path} is a finite path in $\mathbb R^2$ that starts at a point in $\mathbb Z^2$ and uses unit north (i.e., $(0,1)$) steps and unit east (i.e., $(1,0)$) steps; we consider lattice paths that are translations of each other to be the same. We denote north steps by $\text{N}$ and east steps by $\text{E}$, and we identify lattice paths with finite words over the alphabet $\{\text{N},\text{E}\}$. We frequently use superscripts to denote concatenation of words; for instance, $(\text{NE}^2)^3=\text{NEENEENEE}$. Fix nonnegative integers $n\leq\ell$ and a lattice path $\nu$ that uses $n$ north steps and $\ell-n$ east steps, and let $\Tam(\nu)$ denote the collection of lattice paths that lie weakly above $\nu$ and have the same endpoints as $\nu$. 

Given a lattice point $v=(x,y)$ that lies weakly above $\nu$ and satisfies $0\leq x\leq\ell-n$ and $0\leq y\leq n$, we define the \dfn{horizontal distance} of $v$ to be the largest integer $d$ such that $(x+d,y)$ is on $\nu$. Now suppose $\mu\in\Tam(\nu)$ and $v$ is a lattice point on $\mu$ that is immediately preceded by an east step and immediately followed by a north step in $\mu$. Let $v'$ be the first lattice point on $\mu$ that appears after $v$ and has the same horizontal distance as $v$. Let $\text{D}_{[v,v']}$ be the subpath of $\mu$ whose endpoints are $v$ and $v'$. There are lattice paths $\text{X}$ and $\text{Y}$ such that $\mu=\text{X}\text{E}\text{D}_{[v,v']}\text{Y}$. Let $\mu'=\text{X}\text{D}_{[v,v']}\text{E}\text{Y}$. Then $\mu\lessdot\mu'$ is a cover relation in $\Tam(\nu)$. (See \Cref{fig:Tam(nu)_cover}.)

\begin{figure}[ht]
  \begin{center}{\includegraphics[height=2.33cm]{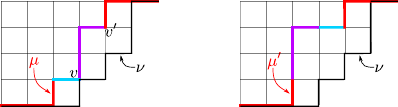}}
  \end{center}
\caption{The lattice paths $\mu=\text{EENENNENEE}$ (left) and $\mu'=\text{EENNNEENEE}$ (right) form the cover relation $\mu\lessdot \mu'$ in $\Tam(\nu)$, where $\nu=\text{EEENENENNE}$. }\label{fig:Tam(nu)_cover}
\end{figure}

The cover relations defined in the previous paragraph make $\Tam(\nu)$ into a poset. Pr\'eville-Ratelle and Viennot \cite{PrevilleViennot} showed that $\Tam(\nu)$ is a lattice called the \dfn{$\nu$-Tamari lattice}. The $n$-th \dfn{$m$-Tamari lattice} $\Tam_n(m)$, which was originally introduced by Bergeron and Pr\'eville-Ratelle \cite{Bergeron}, is $\Tam((\text{NE}^m)^n)$. Finally, $\Tam_n(1)=\Tam((\text{NE})^n)$ is the $n$-th \dfn{Tamari lattice}, which was originally defined by Tamari \cite{Tamari}; we denote it simply by $\Tam_n$. Pr\'eville-Ratelle and Viennot actually proved that the $\nu$-Tamari lattice is an interval of some Tamari lattice; since Tamari lattices are known to be trim, it follows from \Cref{prop:trim_interval} that $\Tam(\nu)$ is trim. \Cref{fig:Tam_3(2)} portrays $\Tam_3(2)$. 

\begin{figure}[ht]
  \begin{center}{\includegraphics[height=14cm]{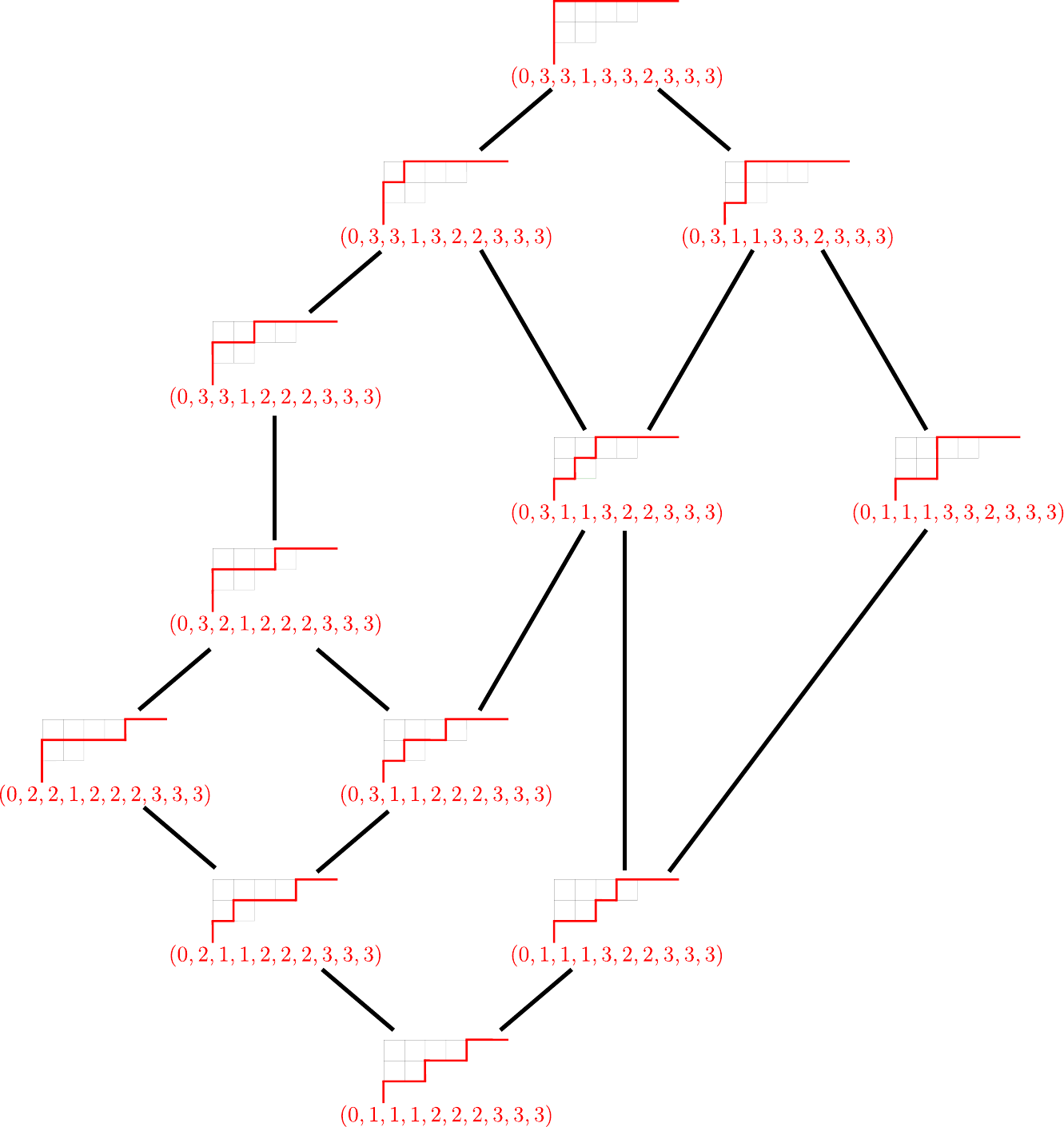}}
  \end{center}
  \caption{The Hasse diagram of $\Tam_3(2)=\Tam(\nu)$, where $\nu=(\text{NE}^2)^3$. Each lattice path appears above its associated $\nu$-bracket vector.}\label{fig:Tam_3(2)}
\end{figure}

If we draw the lattice path $\nu$ in the square grid, then we can consider the set $\Cells(\nu)$ of all unit grid cells that lie directly north of one of the east steps in $\nu$ and lie directly west of one of the north steps in $\nu$. Define a partial order $\leq$ on $\Cells(\nu)$ so that for all $\Box_1,\Box_2\in\Cells(\nu)$, we have $\Box_1\leq\Box_2$ if and only if $\Box_1$ is weakly southwest of $\Box_2$. For example, \Cref{fig:cells} shows $\Cells(\text{ENEEENNEENNE})$ with each cover relation $\Box_1\lessdot\Box_2$ represented by an arrow from $\Box_1$ to $\Box_2$. 

\begin{figure}[ht]
  \begin{center}{\includegraphics[height=2.869cm]{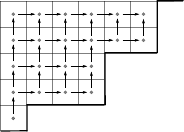}}
  \end{center}
  \caption{The set $\Cells(\nu)$, where $\nu=\text{ENEEENNEENNE}$. Each cover relation $\Box_1\lessdot\Box_2$ is represented by an arrow from $\Box_1$ to $\Box_2$.}\label{fig:cells}
\end{figure}

\begin{proposition}\label{prop:Tam_nu_Cells}
For any lattice path $\nu$, the Galois poset $\Pos(\Tam(\nu))$ is isomorphic to $\Cells(\nu)$. 
\end{proposition}

The previous proposition provides a simple description of the spine of $\Tam(\nu)$ because, according to \eqref{eq:spine_isomorphism}, $\spine(\Tam(\nu))\cong J(\Pos(\Tam(\nu)))$. Before proving it, let us quickly see how it implies \Cref{thm:nu-Tamari}. For each $\ell\geq 1$, choose a lattice path $\nu^{(\ell)}$ with $n_\ell$ north steps and $\ell-n_\ell$ east steps, and assume that the limit $\overline n=\lim\limits_{\ell\to\infty}\frac{1}{\ell}n_\ell$ exists. Because $\Cells(\nu^{(\ell)})$ is clearly a subposet of the rectangle $R_{n_\ell,\ell-n_\ell}$, it follows from \Cref{thm:rectangle,cor:subposet} that \[\mathcal E(J(\Cells(\nu^{(\ell)})))\leq\frac{1}{p}\left(1+2\sqrt{(1-p)\overline n(1-\overline n)}\right)\ell+o(\ell).\] If we assume \Cref{prop:Tam_nu_Cells}, then we can deduce from \eqref{eq:spine_isomorphism} and \Cref{thm:spine} that \[\mathcal E(\Tam(\nu^{(\ell)}))\leq\mathcal E(\spine(\Tam(\nu^{(\ell)})))=\mathcal E(J(\Cells(\nu^{(\ell)}))),\] thereby proving \Cref{thm:nu-Tamari}.   

Let $\nu$ be a lattice path that starts at $(0,0)$ and ends at $(\ell-n,n)$.
In order to prove \Cref{prop:Tam_nu_Cells}, we will need an alternative way of thinking about the $\nu$-Tamari lattice. Let ${\bf h}(\nu)=(h_0(\nu),\ldots,h_{\ell}(\nu))$ be the vector obtained by reading the heights (i.e., $y$-coordinates) of the lattice points on $\nu$ in the order they appear in $\nu$. For $0\leq k\leq n$, let $f_k$ be the maximum index such that $h_{f_k}(\nu)=k$. A \dfn{$\nu$-bracket vector} is an integer vector $\mathsf{b}=(\mathsf b_0,\ldots,\mathsf b_{\ell})$ satisfying the following conditions:
\begin{enumerate}[(I)]
\item $\mathsf b_{f_k}=k$ for all $0\leq k\leq n$; 
\item\label{Item2} $h_i(\nu)\leq \mathsf b_i\leq n$ for all $0\leq i\leq \ell$;
\item\label{Item3} $\mathsf b$ avoids the pattern $121$. 
\end{enumerate} 
The third condition means that there do not exist indices $i_1<i_2<i_3$ such that $\mathsf b_{i_1}=\mathsf b_{i_3}<\mathsf b_{i_2}$. 
Ceballos, Padrol, and Sarmiento \cite{Ceballos2} proved that for each lattice path $\mu\in\Tam(\nu)$, there is a unique $\nu$-bracket vector ${\bf v}_\nu(\mu)$ such that for every $k\in\{0,\ldots,n\}$, the number of lattice points on $\mu$ with height $k$ is equal to the number of occurrences of $k$ in ${\bf v}_\nu(\mu)$. (In particular, ${\bf v}_\nu(\nu)={\bf h}(\nu)$.) This correspondence is illustrated in \Cref{fig:Tam_3(2)}. 

There is an obvious partial order on the set of $\nu$-bracket vectors given by componentwise comparison: $\mathsf{b}\leq\mathsf b'$ if and only if $\mathsf b_i\leq\mathsf b_i'$ for all $0\leq i\leq\ell$. This poset has a meet operation given by the componentwise minimum: \[\mathsf b\wedge\mathsf b'=(\min\{\mathsf b_0,\mathsf b'_0\},\ldots,\min\{\mathsf b_\ell,\mathsf b'_\ell\}).\] Indeed, it is straightforward to check that the componentwise minimum of two $\nu$-bracket vectors is a $\nu$-bracket vector.
It is well known that a finite poset with a meet operation and a unique maximal element is a lattice. This poset on $\nu$-bracket vectors has ${\bf v}_\nu(\text{N}^n\text{E}^{\ell-n})$ as its unique maximal element, so it is a lattice. In fact, Ceballos, Padrol, and Sarmiento \cite{Ceballos2} proved that the map $\mu\mapsto{\bf v}_\nu(\mu)$ is a lattice isomorphism from $\Tam(\nu)$ to this lattice on the set of $\nu$-bracket vectors. Therefore, in the remainder of this subsection, we will abuse notation and identify $\Tam(\nu)$ (via this isomorphism) with the lattice of $\nu$-bracket vectors under the componentwise order. 

Note that we $0$-index our vectors. For instance, we would say that $5$ appears in positions $0$ and $3$ in the vector $(5,2,1,5,1,3)$. We always write $\mathsf b_i$ for the entry in position $i$ of a vector $\mathsf b$. 

Define a \dfn{descent} of a $\nu$-bracket vector $\mathsf b$ to be an index $i\in\{0,\ldots,\ell-1\}$ such that $\mathsf b_i>\mathsf b_{i+1}$. Let $\Des(\mathsf b)$ denote the set of descents of $\mathsf b$. According to \cite[Proposition~4.4]{DefantMeeting}, the number of descents of $\mathsf b$ is equal to the number of elements covered by $\mathsf b$ in $\Tam(\nu)$. In particular, $\mathsf b$ is join-irreducible if and only if $\lvert\Des(\mathsf b)\rvert=1$. 

Suppose $1\leq k\leq n$, $f_{k-1}<i<f_k$, and $h_i(\nu)+1\leq m\leq n$ (note that $h_i(\nu)=k$).
Let $\mathfrak b^{i,m}$ and $\mathfrak c^{i,m}$ be the $\nu$-bracket vectors such that for all $0\leq r\leq\ell$, the entries in position $r$ of $\mathfrak b^{i,m}$ and $\mathfrak c^{i,m}$ are \
\begin{equation}\label{eq:bim}
\mathfrak b^{i,m}_r=\begin{cases} m & \mbox{ if } f_{k-1}+1\leq r\leq i; \\
h_r(\nu) & \mbox{ otherwise.} \end{cases}
\end{equation}
and 
\begin{equation}\label{eq:kappa}
\mathfrak c^{i,m}_r=\begin{cases} s & \mbox{ if } r=f_s \text{ for some }s\in\{0,\ldots,n\}; \\
m-1 & \mbox{ if }i\leq r<f_{m-1}\text{ and }r\not\in\{f_0,\ldots,f_n\}; \\ 
n & \mbox{ otherwise}. \end{cases}
\end{equation}

\begin{example}\label{exam:nu-Tamari1}
Suppose $\nu=\text{ENEEENNEENNE}$. Then ${\bf h}(\nu)=(0,0,1,1,1,1,2,3,3,3,4,5,5)$, so $f_0=1$, $f_1=5$, $f_2=6$, $f_3=9$, $f_4=10$, and $f_5=12$. Setting $i=3$ and $m=4$ in \eqref{eq:bim} and \eqref{eq:kappa}, we find that 
\[\mathfrak b^{3,4}=\begin{array}{l}\includegraphics[height=3.565cm]{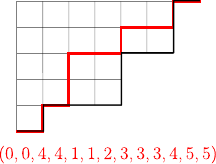}\end{array}\qquad\text{and}\qquad \mathfrak c^{3,4}=\begin{array}{l}\includegraphics[height=3.565cm]{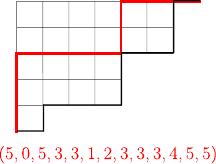}\end{array}\] (where we are representing elements of $\Tam(\nu)$ as both lattices paths and $\nu$-bracket vectors).   
\end{example}

We will need the notation and terminology from \Cref{subsec:trim_basics}. The lattice path corresponding to $\mathfrak c^{i,m}$ only has a single east step that is immediately followed by a north step, so it follows from the above description of cover relations in $\Tam(\nu)$ (in terms of lattice paths) that $\mathfrak c^{i,m}$ is meet-irreducible in $\Tam(\nu)$. Moreover, it is straightforward to see that every meet-irreducible element of $\Tam(\nu)$ arises in this way. In other words, 
\begin{equation}\label{eq:M_Tam}
\mathcal M_{\Tam(\nu)}=\{\mathfrak c^{i,m}:i\in\{0,\ldots,\ell\}\setminus\{f_0,\ldots,f_n\}, h_i(\nu)+1\leq m\leq n\}.
\end{equation}
The vector $\mathfrak b^{i,m}$ is join-irreducible because its only descent is $i$. Because $\Tam(\nu)$ is trim, we know from \Cref{subsec:trim_basics} that there is a bijection $\kappa=\kappa_{\Tam(\nu)}\colon \mathcal J_{\Tam(\nu)}\to\mathcal M_{\Tam(\nu)}$. In light of \eqref{eq:M_Tam}, this implies that every join-irreducible element of $\Tam(\nu)$ is of the form $\mathfrak b^{i,m}$. In other words, 
\begin{equation}\label{eq:J_Tam}
\mathcal J_{\Tam(\nu)}=\{\mathfrak b^{i,m}:i\in\{0,\ldots,\ell\}\setminus\{f_0,\ldots,f_n\}, h_i(\nu)+1\leq m\leq n\}.
\end{equation}

In order to understand the Galois poset $\Pos(\Tam(\nu))$, we must understand the Galois graph $\Gal(\Tam(\nu))$. This, in turn, requires us to understand the bijection $\kappa\colon \mathcal J_{\Tam(\nu)}\to\mathcal M_{\Tam(\nu)}$. 

\begin{lemma}\label{lem:kappa}
Suppose $1\leq k\leq n$, $f_{k-1}<i<f_k$, and $h_i(\nu)+1\leq m\leq n$. Then \[\kappa(\mathfrak b^{i,m})=\mathfrak c^{i,m}.\] 
\end{lemma}

\begin{proof}
By the definition of $\kappa$ in \eqref{eq:kappa_definition}, we just need to demonstrate that $\mathfrak b^{i,m}\wedge\mathfrak c^{i,m}$ is covered by $\mathfrak b^{i,m}$ and that $\mathfrak b^{i,m}\vee\mathfrak c^{i,m}$ covers $\mathfrak c^{i,m}$. The meet operation is given by the componentwise minimum, so it follows from \eqref{eq:bim} and \eqref{eq:kappa} that \[(\mathfrak b^{i,m}\wedge\mathfrak c^{i,m})_r=\begin{cases} m & \mbox{ if } f_{k-1}+1\leq r\leq i-1; \\ 
m-1 & \mbox{ if } r=i; \\
h_r(\nu) & \mbox{ otherwise.} \end{cases}\] It is immediate from this description that $\mathfrak b^{i,m}\wedge\mathfrak c^{i,m}$ is covered by $\mathfrak b^{i,m}$. 

Now let $\mathsf b=(\mathfrak c^{i,m})^*$ be the unique element that covers $\mathfrak c^{i,m}$. Note that $\mathfrak b^{i,m}\not\leq\mathfrak c^{i,m}$ because $\mathfrak b^{i,m}_i=m>m-1=\mathfrak c^{i,m}_i$. Hence, to prove that $\mathfrak b^{i,m}\vee\mathfrak c^{i,m}=\mathsf b$, we just need to show that $\mathfrak b^{i,m}\leq\mathsf b$. For every index $q\in\{0,\ldots,\ell\}\setminus\{i\}$, we have $\mathfrak b^{i,m}_q\leq\mathfrak c^{i,m}_q\leq\mathsf b_q$ by \eqref{eq:bim} and \eqref{eq:kappa}. Thus, we are left to show that $\mathfrak b^{i,m}_i\leq\mathsf b_i$. 

Since $\mathsf b>\mathfrak c^{i,m}$, we know that $\mathsf b_i\geq\mathfrak c^{i,m}_i=m-1$ and that there is an index $r$ such that $\mathsf b_r>\mathfrak c^{i,m}_r$. Appealing to \eqref{eq:kappa} and the definition of a $\nu$-bracket vector, we find that $i\leq r<f_{m-1}$, $r\not\in\{f_0,\ldots,f_n\}$, and $\mathfrak c^{i,m}_r=m-1$. If $\mathsf b_i=m-1$, then $\mathsf b_i=\mathsf b_{f_{m-1}}=m-1<\mathsf b_r$, so the entries of $\mathsf b$ in positions $i,r,f_{m-1}$ form a $121$-pattern, contradicting the definition of a $\nu$-bracket vector. It follows that $\mathsf b_i\geq m=\mathfrak b^{i,m}_i$, as desired. 
\end{proof}

\begin{proof}[Proof of \Cref{prop:Tam_nu_Cells}]
Assume that $\nu$ starts at $(0,0)$ and ends at $(\ell-n,n)$. Suppose $1\leq k\leq n$, $f_{k-1}<i<f_k$, and $h_i(\nu)+1\leq m\leq n$ (note that $h_i(\nu)=k$). The $i$-th step in $\nu$ is an east step at height $k$; let $\Box^{i,m}$ be the grid cell that lies above that east step and has its center at height $m-\frac{1}{2}$. Every grid cell in $\Cells(\nu)$ arises uniquely in this way; that is, 
\[\Cells(\nu)=\{\Box^{i,m}:i\in\{0,\ldots,\ell\}\setminus\{f_0,\ldots,f_n\}, h_i(\nu)+1\leq m\leq n\}.\] Recall from \Cref{subsec:trim_basics} that the underlying set of the Galois poset $\Pos(\Tam(\nu))$ is $\mathcal J_{\Tam(\nu)}$. It is now immediate from \eqref{eq:J_Tam} that we have a bijection $\Pos(\Tam(\nu))\to\Cells(\nu)$ given by $\mathfrak b^{i,m}\mapsto\Box^{i,m}$; we just need to prove that this map is a poset isomorphism. 

We first show that if $\Box^{i,m}\leq\Box^{i',m'}$ in $\Cells(\nu)$ (so $i\leq i'$ and $m\leq m'$), then $\mathfrak b^{i,m}\preceq\mathfrak b^{i',m'}$ in $\Pos(\Tam(\nu))$; it suffices to do so when either $i=i'$ or $m=m'$. If $i=i'$, then $m-1\geq h_{i}(\nu)=h_{i'}(\nu)$. If $m=m'$, then $m-1=m'-1\geq h_{i'}(\nu)$. In either case, $m-1\geq h_{i'}(\nu)$. Recall that $f_{h_{i'}(\nu)}$ is the largest position in which the vector ${\bf h}(\nu)$ has the entry $h_{i'}(\nu)$. Since $h_{i'}(\nu)$ is, by definition, the entry in position $i'$ of ${\bf h}(\nu)$, it follows that $i'\leq f_{h_{i'}(\nu)}$. Hence, $i'\leq f_{m-1}$. This shows that $i\leq i'\leq f_{m-1}$, so we know by \eqref{eq:kappa} and \Cref{lem:kappa} that the entry in position $i'$ of $\kappa(\mathfrak b^{i,m})$ is $m-1$. On the other hand, \eqref{eq:bim} tells us that the entry in position $i'$ of $\mathfrak b^{i',m'}$ is $m'$, which is strictly greater than $m-1$. This proves that $\mathfrak b^{i',m'}\not\leq\kappa(\mathfrak b^{i,m})$ in $\Tam(\nu)$. Hence, there is an arrow $\mathfrak b^{i',m'}\to\mathfrak b^{i,m}$ in the Galois graph $\Gal(\Tam(\nu))$, which implies that $\mathfrak b^{i,m}\preceq\mathfrak b^{i',m'}$ in $\Pos(\Tam(\nu))$, as desired. 

We now show that if $\mathfrak b^{i,m}\preceq\mathfrak b^{i',m'}$ in $\Pos(\Tam(\nu))$, then $\Box^{i,m}\leq\Box^{i',m'}$ in $\Cells(\nu)$ (i.e., $i\leq i'$ and $m\leq m'$); it suffices to do so when there is an arrow $\mathfrak b^{i',m'}\to\mathfrak b^{i,m}$ in $\Gal(\Tam(\nu))$. Thus, $\mathfrak b^{i',m'}\not\leq\kappa(\mathfrak b^{i,m})$ in $\Tam(\nu)$. This means that there is some index $r$ such that $(\kappa(\mathfrak b^{i,m}))_r<\mathfrak b^{i',m'}_r$. Let $k'$ be the index such that $f_{k'-1}<i'<f_{k'}$. We know from the definition of a $\nu$-bracket vector that $(\kappa(\mathfrak b^{i,m}))_r\geq h_r(\nu)$, so $\mathfrak b^{i',m'}_r>h_r(\nu)$. According to \eqref{eq:bim}, we have $\mathfrak b^{i',m'}_r=m'$ and $f_{k'-1}+1\leq r\leq i'$. 
Also, we have $\mathfrak b^{i',m'}_r\leq n$, so $(\kappa(\mathfrak b^{i,m}))_r\neq n$. Since $r\not\in\{f_0,\ldots,f_n\}$ (because $f_{k'-1}<r\leq i'<f_{k'}$), it follows from \eqref{eq:kappa} and \Cref{lem:kappa} that $(\kappa(\mathfrak b^{i,m}))_r=m-1$ and $i\leq r< f_{m-1}$. We have shown that $i\leq r\leq i'$ and $m-1=(\kappa(\mathfrak b^{i,m}))_r<\mathfrak b^{i',m'}_r=m'$. Hence, $i\leq i'$ and $m\leq m'$, so $\Box^{i,m}\leq\Box^{i',m'}$.  
\end{proof}

\section{Tamari Lattices}\label{sec:Tamari}

A permutation $x\in S_n$ is called \dfn{$312$-avoiding} if there do not exist indices $i_1<i_2<i_3$ such that $x(i_2)<x(i_3)<x(i_1)$. Let $\Av_n(312)$ denote the set of $312$-avoiding permutations in $S_n$. Then $\Av_n(312)$ is precisely the set of $c$-sortable elements of $S_n$, where $c=s_1s_2\cdots s_{n-1}\in A_{n-1}=S_n$, so $\Camb_c$ is the sublattice of the weak order on $S_n$ consisting of $312$-avoiding permutations. The Tamari lattice $\Tam_n$ is isomorphic to this Cambrian lattice $\Camb_c$ \cite{ReadingCambrian}, so it follows from \Cref{thm:CambA} that $\mathcal E(\Tam_n)\leq \frac{1}{p}\left(2+2\sqrt{1-p}\right)n+o(n)$. This bound is also a special case of \Cref{thm:nu-Tamari}. Because Tamari lattices are such important objects, we wish to improve upon this bound by proving \Cref{thm:Tamari}. To do so, we will analyze Tamari lattices directly (without appealing to \Cref{thm:spine}), using the interpretation of the $n$-th Tamari lattice as $\Av_n(312)$ under the weak order. 

\Cref{fig:TamariPlots} shows the plots of twelve of the permutations that we obtained while running $\bU_{\Av_{400}(312)}$ with $p=1/2$ starting with the decreasing permutation. These plots illustrate behavior that appears to be typical: the permutations become ``mostly sorted'' after a small number of steps, but there are a few entries that straggle behind and take a long time to reach their correct positions. Experiments also suggest that the number of steps needed to go from the decreasing permutation to the identity permutation in $\bU_{\Av_n(312)}$ has a relatively high variance, and this makes the analysis difficult.  

\begin{figure}[ht]
  \begin{center}{\includegraphics[width=\linewidth]{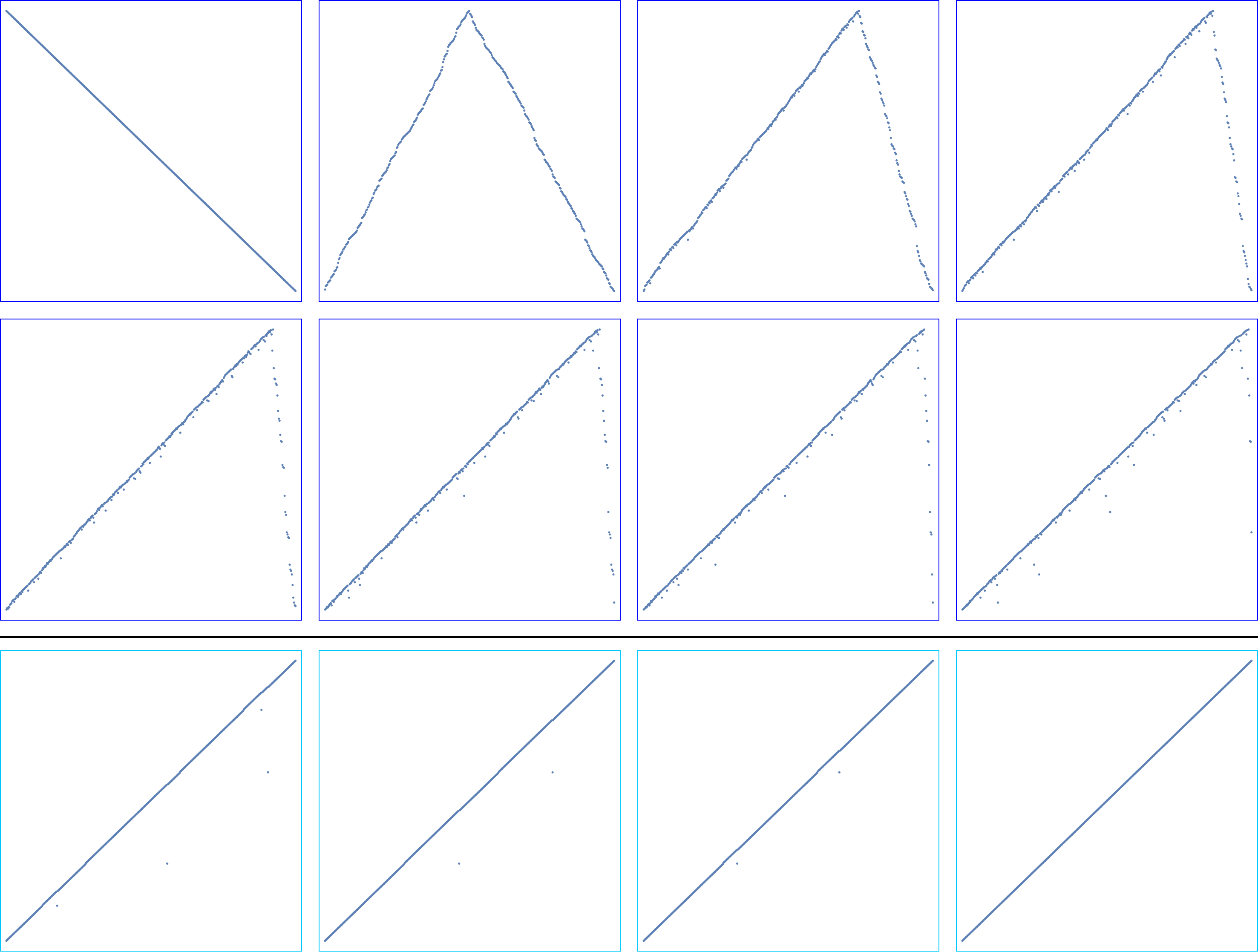}}
  \end{center}
  \caption{We ran the Ungarian Markov chain $\bU_{\Av_{400}(312)}$ with $p=1/2$ starting with the decreasing permutation. Here, we show the plots of the permutations that the Markov chain reached at times $0,1,2,3,4,5,6,7,90,180,270,360$ (read from left to right, row by row). Note that the permutations get ``mostly sorted'' within the first few steps. To indicate that the permutations in the bottom row follow a time skip, we have colored their bounding boxes with a slightly different hue.}\label{fig:TamariPlots}
\end{figure}

In order to analyze $\bU_{\Av_n(312)}$, we first need an explicit combinatorial description of how to apply a random Ungar move to a permutation in $\Av_n(312)$.
Consider $x\in S_n$.
If there exist indices $i$ and $i'$ such that $i +1 < i'$ and $x(i+1)<x(i')<x(i)$, then we can perform an \dfn{allowable swap} by swapping the entries $x(i)$ and $x(i+1)$ (i.e., replacing $x$ by $xs_i$).
Let $\pi_\downarrow(x)$ denote the permutation obtained by starting with $x$ and repeatedly applying allowable swaps until no more allowable swaps can be performed. The element $\pi_\downarrow(x)$ is well defined (i.e., does not depend on the sequence of allowable swaps) and is in $\Av_n(312)$ \cite{ReadingCambrian}. Thus, we have a \emph{projection} $\pi_\downarrow\colon S_n\to\Av_n(312)$. Note that $\pi_\downarrow(x)=x$ if and only if $x$ is $312$-avoiding. 

The first author showed \cite[Equation~(1)]{DefantMeeting} that $\Pop_{\Av_n(312)}(x)=\pi_\downarrow(\Pop_{S_n}(x))$ for every $x\in \Av_n(312)$. In other words, if $x\in\Av_n(312)$, then applying a maximal Ungar move to $x$ within the Tamari lattice $\Av_n(312)$ is equivalent to applying a maximal Ungar move to $x$ within the weak order lattice $S_n$ and then applying the projection $\pi_\downarrow$. The exact same argument (which we omit) shows that applying a random Ungar move to $x$ within $\Av_n(312)$ is equivalent to applying a random Ungar move to $x$ within $S_n$ and then applying $\pi_\downarrow$. Thus, a random Ungar move within $\Av_n(312)$ will send $x$ to $\pi_\downarrow(xw_0(T))$, where $T$ is a random subset of $\Des(x)$ (with each descent of $x$ added to $T$ with probability $p$) and $w_0(T)$ is the maximal element of the parabolic subgroup of $S_n$ generated by $\{s_i:i\in T\}$. 

As in \Cref{sec:Weak}, we write $\DB(w)$ for the set $\{w(i+1):i\in\Des(w)\}$ of descent bottoms of a permutation $w\in S_n$. Let $\LRMax(w)$ denote the set $\{w(i):w(j)<w(i)\text{ for all }1\leq j\leq i-1\}$ of \dfn{left-to-right maxima} of $w$. If $w'\leq w$ in the weak order, then $\LRMax(w)\subseteq\LRMax(w')$. A simple but useful fact that we will exploit is that if $w$ is $312$-avoiding, then $\DB(w)$ and $\LRMax(w)$ form a partition of $[n]$. This implies that if $w'\leq w$ in $\Av_n(312)$ (for instance, if $w'$ is obtained by applying an Ungar move to $w$), then $\DB(w')\subseteq\DB(w)$. 

For integers $\beta\in[n-1]$ and $j\geq 0$, let $X_j^{(\beta)}$ be a Bernoulli random variable with expected value $p$; assume that all of these random variables for different choices of $\beta$ and $j$ are independent. We can simulate $\bU_{\Av_n(312)}$ using these random variables as follows. Starting with the decreasing permutation $\sigma_0=n(n-1)\cdots 1$, we will create a sequence $\sigma_0,\sigma_1,\ldots$ of $312$-avoiding permutations. Suppose we have already generated the permutations $\sigma_0,\ldots,\sigma_t$. 
Let \vspace{-.8cm}\[T_t=\left\{i\in\Des(\sigma_t):X_{t}^{(\sigma_t(i+1))}=1\right\}.\]
We define $\sigma_{t+1}$ to be the permutation obtained by applying the random Ungar move (within $\Av_n(312)$) corresponding to the set $T_t$ to $\sigma_t$; that is, $\sigma_{t+1}=\pi_\downarrow(\sigma_tw_0(T_t))$. Note that for permutations $w,w'\in\Av_n(312)$, the conditional probability $\mathbb P(\sigma_{t+1}=w'\mid \sigma_t=w)$ is equal to the transition probability $\mathbb P(w\to w')$ in $\bU_{\Av_n(312)}$. 

Let $N$ be the unique integer such that $\sigma_{N-1}\neq\sigma_N=12\cdots n$. Then $\sigma_t=12\cdots n$ for all $t\geq N$. For $\beta\in[n-1]$, let $q(\beta)$ be the largest integer $t$ such that $\beta\in\DB(\sigma_t)$.
As the sequence $\sigma_0,\sigma_1,\dots$ is weakly decreasing in the weak order, we have the chain of containments $\DB(\sigma_0)\supseteq\DB(\sigma_1)\supseteq\cdots$. Thus, $\beta\in\DB(\sigma_t)$ for all $t \leq q(\beta)$.
If $\beta<\beta'$ and $\beta$ appears to the left of $\beta'$ in a permutation $\sigma_t$, then $\beta$ must also appear to the left of $\beta'$ in all of the permutations $\sigma_{t+1},\sigma_{t+2},\ldots$. Consequently, 
\begin{equation}\label{eq:3Tam}
\sum_{j=0}^{q(\beta)}X_j^{(\beta)}\leq n-\beta.
\end{equation}
Observe that $N-1=\max\limits_{1\leq \beta\leq n-1}q(\beta)$. 

For each $\beta\in[n-1]$, let $F_\beta$ be the smallest nonnegative integer $j$ such that $X_j^{(\beta)}=1$. Then $F_1,\ldots,F_{n-1}$ are independent geometric random variables, each with expected value $1/p$. For each real number $\zeta\in[1/n,1]$, let $A_\zeta$ be the event that $F_\beta\leq\max\{F_i:\beta+1\leq i\leq\beta+\zeta n\}$ for all integers $\beta$ satisfying $1\leq\beta< (1-\zeta)n$. 

\begin{lemma}\label{lem:event_B}
Suppose $B$ is an event that only depends on the random variables $X_j^{(\beta)}$ with $1\leq \beta\leq n-1$ and $0\leq j\leq \sqrt{n}$. If $\xi\in(0,1]$, $n\geq 1/\xi$, and $B\subseteq A_\xi$, then $\mathbb E(N\mid B)\leq \xi n/p+o(n)$. 
\end{lemma}

\begin{proof}
Suppose the event $B$ occurs. In particular, this implies that $A_\xi$ occurs. Consider some $\beta\in[n-1]$. Let $\beta'$ be the entry immediately to the left of $\beta$ in $\sigma_{F_\beta}$. If $1\leq\beta<(1-\xi)n$, then $F_\beta\leq\max\{F_i:\beta+1\leq i\leq\beta+\xi n\}$, so $\beta'\leq\beta+\xi n$. On the other hand, if $\beta\geq(1-\xi)n$, then the inequality $\beta'\leq\beta+\xi n$ holds automatically. Since $X_{F_\beta}^{(\beta)}=1$, the entry $\beta$ appears to the left of $\beta'$ in $\sigma_{F_\beta+1}$. Since $\sigma_{F_\beta+1}$ is $312$-avoiding, all entries that are greater than $\beta$ and appear to the left of $\beta$ in $\sigma_{F_\beta+1}$ must be at most $\beta+\xi n$; hence, there are at most $\xi n-1$ such entries (the $-1$ term comes from the fact that $\beta'$ is not one of these entries). Consequently, $\sum_{j=F_\beta+1}^{q(\beta)}X_j^{(\beta)}\leq\xi n-1$. The definition of $F_\beta$ now guarantees that $\sum_{j=0}^{q(\beta)}X_j^{(\beta)}\leq\xi n$. Hence, $\sum_{\sqrt{n}<j\leq q(\beta)}X_j^{(\beta)}\leq\xi n$. Because $B$ is independent of the variables $X_j^{(\beta)}$ with $j>\sqrt{n}$, this implies that $q(\beta)-\sqrt{n}$ is bounded above by a sum of $\xi n$ independent geometric random variables, each with expected value $1/p$. For $m\geq \frac{\xi n/p}{1-n^{-1/4}}$, we can set $k=\xi n$ and $\gamma=\frac{pm}{\xi n}$ in \Cref{lem:geometric_tail} to find that \[\mathbb P(q(\beta)-\sqrt{n}>m\mid B)\leq \exp\left(-(p/2)m\left(1-\frac{\xi n}{pm}\right)^2\right)\leq\exp\left(-(p/2)mn^{-1/2}\right).\] As $\beta$ was arbitrary, we can use the identity $N-1=\max\limits_{1\leq\beta\leq n-1}q(\beta)$ and a union bound to see that \[\mathbb P(N-1-\sqrt{n}>m\mid B)\leq n\exp\left(-(p/2)mn^{-1/2}\right).\] 
Hence, 
\begin{align*}
\sum_{m\geq \frac{\xi n/p}{1-n^{-1/4}}}\mathbb P(N-1-\sqrt{n}>m\mid B)&\leq\sum_{m\geq \frac{\xi n/p}{1-n^{-1/4}}}n\exp\left(-(p/2)mn^{-1/2}\right) \\ 
&\leq n\frac{\exp\left(-(p/2)n^{-1/2}\frac{\xi n/p}{1-n^{-1/4}}\right)}{1-\exp\left(-(p/2)n^{-1/2}\right)} \\ 
&\leq n\frac{\exp\left(-\xi n^{1/2}/2\right)}{1-\exp\left(-(p/2)n^{-1/2}\right)} \\ 
&=O\left(n^{3/2}e^{-\xi n^{1/2}/2}\right) \\ 
&=o(1).
\end{align*}
We deduce that
\begin{align*}
\mathbb E(N\mid B)&=\sum_{0\leq m<\frac{\xi n/p}{1-n^{-1/4}}+\sqrt{n}+1}\mathbb P(N>m\mid B)+\sum_{m\geq\frac{\xi n/p}{1-n^{-1/4}}+\sqrt{n}+1}\mathbb P(N>m\mid B) \\ 
&\leq\frac{\xi n/p}{1-n^{-1/4}}+\sqrt{n}+2+\sum_{m\geq \frac{\xi n/p}{1-n^{-1/4}}}\mathbb P(N-1-\sqrt{n}>m\mid B) \\ 
&=\xi n/p+o(n). \qedhere
\end{align*}
\end{proof}

Now let $A'$ be the event that $\max\{F_1,\ldots,F_{n-1}\}\leq\sqrt{n}$. \Cref{lem:event_B} yields the following corollaries. 

\begin{corollary}\label{cor:zeta1}
For $\zeta\in(0,1]$ and $n\geq 1/\zeta$, we have $\mathbb E(N\mid A_\zeta\cap A')\leq\zeta n/p+o(n)$. 
\end{corollary}

\begin{proof}
The event $A_\zeta\cap A'$ only depends on the random variables $X_j^{(\beta)}$ with $\beta\in[n-1]$ and $0\leq j\leq\sqrt{n}$, so we can set $\xi=\zeta$ in \Cref{lem:event_B} to deduce the desired result. 
\end{proof}

\begin{corollary}\label{cor:zeta2}
We have $\mathbb P(\neg A')\mathbb E(N\mid \neg A')=O(n^2e^{-p\sqrt{n}/3})$. 
\end{corollary}

\begin{proof}
For each $\beta\in[n-1]$, we can set $k=1$ and $\gamma=p\sqrt{n}$ in \Cref{lem:geometric_tail} to see that
\[ \mathbb P(F_\beta>\sqrt{n})\leq e^{-\frac{p\sqrt{n}}{2}(1-(p\sqrt{n})^{-1})^2}=O(e^{-p\sqrt{n}/3}). \]
Taking a union bound over all $\beta\in[n-1]$ shows that $\mathbb P(\neg A')=O(ne^{-p\sqrt{n}/3})$.  

To complete the proof, we just need to show that $\mathbb E(N\mid\neg A')=O(n)$. The event $\neg A'$ only depends on the random variables $X_j^{(\beta)}$ with $\beta\in[n-1]$ and $0\leq j\leq\sqrt{n}$. Also, the event $A_1$ always occurs vacuously, so $(\neg A')\subseteq A_1$. Setting $\xi=1$ in \Cref{lem:event_B}, we derive the estimate $\mathbb E(N\mid\neg A')\leq n/p+o(n)=O(n)$. 
\end{proof}

\begin{corollary}\label{cor:zeta3}
For $\zeta\in(0,1]$ and $n\geq 1/\zeta$, we have $\mathbb E(N\mid A'\setminus A_\zeta)\leq n/p+o(n)$. 
\end{corollary}

\begin{proof}
The event $A'\setminus A_\zeta$ only depends on the random variables $X_j^{(\beta)}$ with $\beta\in[n-1]$ and $0\leq j\leq\sqrt{n}$. The event $A_1$ always occurs vacuously, so $(A'\setminus A_\zeta)\subseteq A_1$. Setting $\xi=1$ in \Cref{lem:event_B} yields the estimate $\mathbb E(N\mid A'\setminus A_\zeta)\leq n/p+o(n)$. 
\end{proof}

Let $G_1,G_2,\ldots$ be independent geometric random variables, each with expected value $1/p$. Recall from \Cref{sec:intro} that Bruss and O'Cinneide \cite{Bruss} proved that $\lim\limits_{n\to\infty}(\rho_p(n)-\Upsilon_p(n))=0$, where \[\Upsilon_p(x)=\begin{cases} px\displaystyle\sum_{k\in\mathbb Z}(1-p)^ke^{-(1-p)^kx} & \mbox{ if } p<1; \\
0 & \mbox{ if }p=1 \end{cases}\] 
and $\rho_p(n)$ is the probability that there is a unique integer $i\in[n]$ such that $G_i=\max\{G_1,\ldots,G_n\}$. Moreover, $\Upsilon_p(x)$ is logarithmically periodic, so $\max\limits_{0<x<1}\Upsilon_p(x)=\limsup\limits_{n\to\infty}\rho_p(n)$. Let \[\overline\rho_p=\max_{0<x<1}\Upsilon_p(x)=\limsup\limits_{n\to\infty}\rho_p(n).\] 
We can now complete the proof of \Cref{thm:Tamari}. 

\begin{proof}[Proof of \Cref{thm:Tamari}]
If $p=1$, then \Cref{thm:Tamari} states that $\mathcal E(\Tam_n)=o(n)$, which is certainly true because $\hat 1$ transitions to $\hat 0$ in a single step in $\bU_{\Tam_n}$. Now assume $p<1$. Let $\zeta=\sqrt{\overline\rho_p/(1+\overline\rho_p)}\in(0,1]$, and assume $n\geq 1/\zeta$. If $\beta$ is an integer satisfying $1\leq\beta<(1-\zeta)n$, then $F_\beta, F_{\beta+1},\ldots, F_{\beta+\left\lfloor\zeta n\right\rfloor}$ are independent geometric random variables, each with expected value $1/p$, so \[\mathbb P\left(F_\beta>\max\{F_i:\beta+1\leq i\leq\beta+\zeta n\}\right)=\frac{1}{\left\lfloor\zeta n\right\rfloor+1}\rho_p(n)\leq \frac{1}{\zeta n}(\overline\rho_p+o(1)).\] 
Consequently, \[1-\mathbb P(A_\zeta)\leq \frac{1-\zeta}{\zeta}(\overline\rho_p+o(1)).\] 
Appealing to \Cref{cor:zeta1,cor:zeta2,cor:zeta3}, we compute 
\begin{align*}
\mathcal E(\Tam_n)&=\mathbb E(N) \\
&= \mathbb P(A_\zeta\cap A')\mathbb E(N\mid A_\zeta\cap A')+\mathbb P(\neg A')\mathbb E(N\mid\neg A')+\mathbb P(A'\setminus A_\zeta)\mathbb E(N\mid A'\setminus A_\zeta) \\ 
&\leq \mathbb P(A_\zeta\cap A')(\zeta n/p+o(n))+O(n^2e^{-p\sqrt{n}/3})+\mathbb P(A'\setminus A_\zeta)(n/p+o(n)) \\ 
&=\mathbb P(A_\zeta\cap A')\zeta n/p+\mathbb P(A'\setminus A_\zeta)n/p+o(n) \\ 
&\leq \mathbb P(A_\zeta)\zeta n/p+(1-\mathbb P(A_\zeta))n/p+o(n) \\ 
&=\zeta n/p+(1-\mathbb P(A_\zeta))(1-\zeta)n/p+o(n) \\ 
&\leq \zeta n/p+\left(\frac{1-\zeta}{\zeta}(\overline\rho_p+o(1))\right)(1-\zeta)n/p+o(n) \\ &=\frac{1}{p}\left(\zeta+\frac{(1-\zeta)^2}{\zeta}\overline\rho_p\right)n+o(n). 
\end{align*}
Substituting $\zeta=\sqrt{\overline\rho_p/(1+\overline\rho_p)}$ and simplifying yields \[\mathcal E(\Tam_n)\leq \frac{2}{p}\left(\sqrt{\overline\rho_p(1+\overline\rho_p)}-\overline\rho_p\right)n+o(n),\] as desired. 
\end{proof}

\section{Future Directions}\label{sec:Future} 

\subsection{Better Estimates}
It would be interesting to improve upon the asymptotic upper bounds that we proved for $\mathcal E(L)$ for various lattices $L$. With regard to the weak order on $S_n$, we conjecture that the upper bound in \Cref{thm:weak}, which is on the order of $n\log n$, can be improved to a linear bound. 
\begin{conjecture}
We have $\mathcal E(S_n)=O(n)$. 
\end{conjecture}

In \Cref{thm:Tamari}, we focused on Tamari lattices and improved upon the upper bound that one would obtain from \Cref{thm:CambA} by viewing $\Tam_n$ as a Cambrian lattice of type~$A$. There are other specific interesting families of Cambrian lattices where one could try to make similar improvements. For example, Reading's \dfn{Tamari lattices of type~$B_n$} \cite{ReadingCambrian} are the Cambrian lattices of type $B_n$ given by the Coxeter elements $s_0s_1\cdots s_{n-1}$ and $s_{n-1}\cdots s_1s_0$. Another interesting class of Cambrian lattices worth exploring in more detail consists of the \dfn{bipartite Cambrian lattices}. To construct these, choose a bipartition $X\sqcup Y$ of the Coxeter graph of a finite irreducible Coxeter group $W$ (this is possible because the Coxeter graph is a tree), and let $c_X=\prod_{s\in X}s$ and $c_Y=\prod_{s\in Y}s$. The two bipartite Cambrian lattices of type $W$ are the Cambrian lattices corresponding to the Coxeter elements $c_Xc_Y$ and $c_Yc_X$. 

Most of our results focus on upper bounds for $\mathcal E(L)$. It would be very interesting to have nontrivial lower bounds, especially for specific choices of $L$. Appealing to Ungar's work, we did find a lower bound for $\mathcal E(S_n)$ in \Cref{thm:weak}; can one find an improvement? Can one obtain any nontrivial asymptotic lower bound for $\mathcal E(\Tam_n)$? 

\subsection{The Elements We Reach Along the Way}
\Cref{fig:WeakPlots} shows the plots of some permutations obtained by applying the Ungarian Markov chain $\bU_{S_{400}}$. There are very clear patterns in these plots that we do not know how to explain. Describing these shapes rigorously would be highly substantial. This problem seems similar, at least superficially, to the monumental works \cite{Angel, Angel2, Dauvergne, DauvergneVirag} on random sorting networks.  

More generally, we have the following. 

\begin{problem}
Given a lattice $L$, describe the elements that we are likely to reach when we run the Ungarian Markov chain $\bU_L$ starting at $\hat 1$. 
\end{problem}

There are multiple ways to interpret the previous problem. On the one hand, one could study the probability distribution on $L$ that results from running $\bU_L$ for a prescribed number of steps (starting at $\hat 1$). In a similar but slightly different vein, one could study the hitting probabilities of the elements of $L$. 

\subsection{Variance}
This article has focused exclusively on $\mathcal E(L)$, the expected value of the number of steps needed to go from the top element $\hat 1$ to the bottom element $\hat 0$ in $\bU_L$ (i.e., the expected hitting time of the unique absorbing state). It would be very interesting to say something about the \emph{variance} of this number of steps. Experiments suggest that this variance is relatively small when $L$ is the weak order on $S_n$ but is relatively large when $L$ is the Tamari lattice $\Tam_n$. 

\subsection{Other Lattices}
There are several diverse families of lattices, and we have certainly not explored all of them through the lens of Ungarian Markov chains in this paper. Considering other families would likely be a rich source of future research directions.  

One particular family that stands out is that of \emph{geometric} lattices, which are fundamental because of their numerous examples and their strong connection with matroid theory. A lattice is \dfn{atomic} if each of its elements can be expressed as a join of atoms. If $L$ is a graded lattice, then it has a unique \dfn{rank function} $\text{rk}\colon L\to\mathbb Z_{\geq 0}$ such that $\text{rk}(\hat 0)=0$ and $\text{rk}(y)=\text{rk}(x)+1$ whenever $x\lessdot y$. We say a lattice $L$ is \dfn{semimodular} if it is graded and its rank function satisfies
\[ \text{rk}(x\wedge y)+\text{rk}(x\vee y)\leq\text{rk}(x)+\text{rk}(y) \]
for all $x,y\in L$. A lattice is \dfn{geometric} if it is atomic and semimodular. It is well known that a lattice is geometric if and only if it is isomorphic to the lattice of flats of a matroid. 

Let us say a geometric lattice $L$ is \dfn{primitive} if it cannot be written as a product of smaller geometric lattices. Then $L$ is primitive if and only if it is isomorphic to the lattice of flats of a connected matroid. 

\begin{conjecture}
If $(L^{(n)})_{n\geq 1}$ is a sequence of primitive geometric lattices such that $\lvert L^{(n)}\rvert\to\infty$ as $n\to\infty$, then \[\mathcal E(L^{(n)})=1+o(1).\]
\end{conjecture}

\section*{Acknowledgements}
Colin Defant was supported by the National Science Foundation under Award No.\ 2201907 and by a Benjamin Peirce Fellowship at Harvard University. We are grateful to Noga Alon, Dor Elboim, Dan Romik, and Nathan Williams for very helpful conversations.

\end{document}